\documentclass[12pt]{amsart}
\usepackage{preamble}

\begin{document}

\title[The spum and sum-diameter of graphs]{The spum and sum-diameter of graphs: labelings of sum graphs}
\author[Rupert Li]{Rupert Li}
\address{Massachusetts Institute of Technology, 77 Massachusetts Avenue, Cambridge, MA 02139, USA}
\email{rupertli@mit.edu}

\date{\today}
\begin{abstract}
    A \emph{sum graph} is a finite simple graph whose vertex set is labeled with distinct positive integers such that two vertices are adjacent if and only if the sum of their labels is itself another label.
    The \emph{spum} of a graph $G$ is the minimum difference between the largest and smallest labels in a sum graph consisting of $G$ and the minimum number of additional isolated vertices necessary so that a sum graph labeling exists.
    We investigate the spum of various families of graphs, namely cycles, paths, and matchings.
    We introduce the \emph{sum-diameter}, a modification of the definition of spum that omits the requirement that the number of additional isolated vertices in the sum graph is minimal, which we believe is a more natural quantity to study.
    We then provide asymptotically tight general bounds on both sides for the sum-diameter, and study its behavior under numerous binary graph operations as well as vertex and edge operations.
    Finally, we generalize the sum-diameter to hypergraphs.
\end{abstract}

\maketitle

\section{Introduction}
In 1990, Harary \cite{harary1990sum} defined a \emph{sum graph} to be a graph whose vertices can be labeled with distinct positive integers such that two vertices are adjacent if and only if the sum of their labels is itself another label in the graph.
Not every graph is a sum graph: the vertex with highest label must be isolated, so any graph without isolated vertices cannot be a sum graph.
However, if one adds enough isolated vertices to a graph it will become possible to represent it as a sum graph, and Harary analyzed the minimum number of isolated vertices one must add to an (unlabeled) graph so that it can be represented as a sum graph, which he called the \emph{sum number} $\sigma(G)$.
The sum number of various special families of graphs have been identified, notably including complete graphs $K_n$, cycles $C_n$, and trees.
See Gallian's survey \cite[Table 20]{gallian2018dynamic} for a comprehensive list of known results on the sum number.
Four years later, Harary \cite{harary1994sum} extended his notion of a sum graph to allow distinct integer labels, rather than simply positive integers; the corresponding graph is called an \emph{integral sum graph}, and the corresponding \emph{integral sum number} is denoted $\zeta(G)$.
For precise definitions of all of these concepts, see \cref{section: preliminaries}.

Goodell, Beveridge, Gallagher, Goodwin, Gyori, and Joseph \cite{goodellsum} introduced the notion of \emph{spum}\footnote{the etymology of spum appears to be a portmanteau of ``the \emph{sp}an number of a s\emph{um} graph."}, which is the minimum possible difference between the maximum and minimum labels for a labeling of the sum graph obtained by adding $\sigma(G)$ isolated vertices to $G$.
Singla, Tiwari, and Tripathi \cite{singla2021some} found $\spum(K_n)=4n-6$ for $n \geq 2$, as well as calculated the spum for other families of graphs, such as $K_{1,n}$ and $K_{n,n}$, and bounding the spum of paths $P_n$ and cycles $C_n$.
They also introduce the natural integral variant of spum, where integral spum concerns itself with integral sum graphs, and adds $\zeta(G)$ isolated vertices to $G$ instead of $\sigma(G)$ vertices.
They then find the integral spum for the same families of graphs.

While investigating spum, we came to the conclusion that a modified concept called the sum-diameter is a more fruitful definition.
The \emph{sum-diameter} of a graph $G$, denoted $\sd(G)$, considers labelings of sum graphs consisting of $G$ along with any number of isolated vertices; while spum restricts consideration to using the minimum possible number $\sigma(G)$ of additional isolated vertices, sum-diameter is defined identically, except without requiring the usage of exactly $\sigma(G)$ additional isolated vertices.
Similarly, the integral variant of this is called the \emph{integral sum-diameter}, denoted $\isd(G)$.
While $\spum(G)$ and $\ispum(G)$ are not necessarily related, lifting the restriction of using exactly $\sigma(G)$ or $\zeta(G)$ additional vertices, respectively, yields that for all graphs, $\isd(G) \leq \sd(G)$.
This is to be expected, as expanding the set of labels from the positive integers $\Z_+$ to simply the integers $\Z$ should allow for a more optimal labeling, i.e., a labeling with a smaller difference between maximum and minimum label.
In other words, the requirement that we first optimize the number of isolated vertices before optimizing the labeling is unnatural and impedes natural properties such as $\isd(G) \leq \sd(G)$.
We find this makes bounding $\isd(G)$ far easier than $\ispum(G)$ in numerous cases.

In this paper, we provide a tight general lower bound linear in the number of vertices $n$, as well as an asymptotically tight general upper bound quadratic in $n$ for the sum-diameter of a graph.
A constructive general upper bound similar to this result would be much harder for spum, as $\sigma(G)$ has not been determined in general.
We then expand upon our analysis of the sum-diameter, computing certain special families of graphs and studying its behavior under various binary graph operations and vertex or edge operations.
For completeness, we also determine the spum for various families of graphs, extending the analysis by \cite{singla2021some} while correcting their numerous errors.

In more detail, in \cref{section: preliminaries}, we formally define the necessary concepts and provide a lemma refining the general lower bound on $\spum(G)$ provided in \cite{singla2021some}.
In \cref{section: spum path}, we improve the existing bounds on the spum of paths $\spum(P_n)$.
Then in \cref{section: spum cycle}, we find the exact value of $\spum(C_n)$ for all cycles $C_n$, notably addressing multiple errors in \cite{singla2021some}.
Next, in \cref{section: ispum cycle}, we improve the existing bounds on the integral spum of cycles $\ispum(C_n)$.
Then, in \cref{section: spum ispum nK2} we exactly determine $\spum(nK_2)$ and $\ispum(nK_2)$ for all perfect matchings of $2n$ vertices $nK_2$.

In \cref{section: sum-diameter}, we introduce the notion of the sum-diameter $\sd(G)$ of a graph $G$, as well as the integral sum-diameter $\isd(G)$.
We discuss basic relationships between $\spum(G)$, $\ispum(G)$, $\sd(G)$, and $\isd(G)$, generalize the linear lower bounds of \cite{singla2021some} to (integral) sum-diameter, and provide a quadratic upper bound on $\sd(G)$.
We then demonstrate that this upper bound is tight up to a constant factor, as the maximum of $\sd(G)$ among all graphs $G$ with $n$ vertices is $\Omega(n^2)$.
In \cref{section: uspum K_n}, we augment the previous analysis of $\spum(K_n)$ and $\ispum(K_n)$ by \cite{singla2021some} to exactly identify the values of $\sd(K_n)$ and $\isd(K_n)$, while simultaneously correcting some errors in the original proof for the value of $\spum(K_n)$ and $\ispum(K_n)$.
In \cref{section: isd cycle path}, we bound the sum-diameter and integral sum-diameter of cycles $C_n$ and paths $P_n$.
In \cref{section: sd binary operations}, we provide upper bounds on the sum-diameter under various binary graph operations, in particular including the disjoint union and graph join.
In \cref{section: sd vertex edge addition deletion}, we provide upper bounds on the sum-diameter under various natural graph transformations, namely vertex addition and deletion, which extends to induced subgraphs, as well as edge addition, deletion, and contraction.
We pose an open question concerning the monotonicity of the sum-diameter with respect to induced subgraphs.

In \cref{section: sd hypergraph} we introduce the generalization of sum-diameter to hypergraphs, and generalize our previous general upper and lower bounds to provide preliminary bounds on both sides for arbitrary $k$-uniform hypergraphs.
Finally, we conclude in \cref{section: conclusion} with some remarks on areas for further research and some open questions.

\section{Preliminaries}\label{section: preliminaries}

Let $S \subset \Z$ and $a \in \Z$.
It will be useful to define the following notation:
\begin{align*}
    S+a &= \{s + a \mid s \in S \} \\
    S-a &= \{s - a \mid s \in S \} \\
    aS &= \{ as \mid s \in S \} \\
    -S &= (-1)S.
\end{align*}
Also, define $\range(S) = \max S - \min S$.
Finally, let $\Z_+$ denote the set of positive integers.

We start with the definition of a sum graph and an integral sum graph, introduced by \cite{harary1990sum} and \cite{harary1994sum}, respectively.

\begin{definition}\label{definition: sum graph}
The \emph{induced sum graph} of a set $L \subset \Z$ is the simple graph whose vertex set is $L$ and $(u,v) \in E$ if and only if $u + v \in L$.

A simple graph $G=(V,E)$ is called a \emph{sum graph} if it is isomorphic to the induced sum graph of some set $L \subset \Z_+$.
We call $L$ a set of labels for the sum graph $G$.
We will often not distinguish between the vertices and their respective labels.

Similarly, a graph $G$ is called an \emph{integral sum graph} if such a bijection exists to a set of integers $L \subset \Z$.
\end{definition}

Conversely, any set of positive integers $L$ induces a sum graph with vertex set $V=L$, and any set of integers $L$ induces an integral sum graph.
Notice that the vertex with maximum label in a sum graph must be isolated, so a graph without isolated vertices cannot be a sum graph.
In particular, a connected graph on at least two vertices cannot be a sum graph.
This leads to the following definition.

\begin{definition}\label{definition: sum number}
The \emph{sum number} of a graph $G$, denoted $\sigma(G)$, is the minimum number of isolated vertices that must be added to $G$ in order to yield a sum graph.

Similarly, the \emph{integral sum number} of a graph $G$, denoted $\zeta(G)$, is the minimum number of isolated vertices that must be added to $G$ in order to yield an integral sum graph.
\end{definition}

It was shown in \cite{harary1990sum} that $\sigma(G) \leq |E|$, and as we clearly have $\zeta(G) \leq \sigma(G)$ for all $G$, this implies both $\sigma(G)$ and $\zeta(G)$ are finite.

The sum number and integral sum numbers of various classes of graphs are known; for a collection of such results, we refer the reader to \cite{gallian2018dynamic}.

We now define the spum of a graph $G$.

\begin{definition}\label{definition: spum}
The \emph{spum} of a graph $G$, denoted $\spum(G)$, is the minimum value of $\range(L)$ among all sets of positive integers that induce the sum graph consisting of $G$ with $\sigma(G)$ isolated vertices.

Similarly, the \emph{integral spum} of $G$, denoted $\ispum(G)$, is the minimum value of $\range(L)$ among all sets of integers $L$ that induce the integral sum graph consisting of $G$ with $\zeta(G)$ isolated vertices.
\end{definition}
According to \cite{gallian2018dynamic}, the notion of spum was first introduced in an unpublished paper \cite{goodellsum}, which supposedly also proved that $\spum(K_n)=4n-6$.
This result has been confirmed in \cite{singla2021some}, which also calculated the spum and integral spum for several other classes of graphs, namely including $K_{1,n}$ and $K_{n,n}$, along with bounds for path graphs $P_n$ and cycles $C_n$.

\begin{remark}\label{remark: a1 = min L spum}
For any graph $G$ without isolated vertices and a corresponding sum graph labeling $L$ that achieves $\range(L)=\spum(G)$, let $S\subset L$ be the set of labels assigned to the vertices of $G$.
Notice that $\min S = \min L$, as otherwise removing the label $\min L$ does not change the fact that this labeling induces $G$, while strictly decreasing its range, contradicting the minimality of $\range(L)=\spum(G)$.
\end{remark}

In \cite[Theorem 2.1]{singla2021some}, it was shown that for all connected graphs $G$ with $n$ vertices and maximum and minimum vertex degrees $\Delta$ and $\delta$, respectively, $\spum(G)\geq 2n-(\Delta-\delta)-2$.
The result can be extended to all graphs $G$ without isolated vertices without changing the proof.
We provide the following refinement for the equality case.
\begin{lemma}\label{lemma: spum lower bound equality a1-2a1}
Let $G$ be a graph without isolated vertices, whose order is $n$ and whose maximum and minimum vertex degrees are $\Delta$ and $\delta$, respectively.
If $\spum(G) = 2n-(\Delta-\delta)-2$, then let $L$ be any labeling that achieves this spum value.
Let $S\subset L$ be the set of labels assigned to the vertices of $G$ and let $a_1=\min S$.
Then $[a_1,2a_1] \subseteq S$.
\end{lemma}
\begin{proof}
The proof uses the same structure as the proof of \cite[Theorem 2.1]{singla2021some}, with some additional arguments.

We use the same notation as in the statement of the lemma.
Assume for the sake of contradiction that $[a_1,2a_1] \not\subseteq S$.
Sort $S=\{a_1,\dots,a_n\}$ in increasing order $a_1 < \cdots < a_n$.
Define $S_1 = S \cap [a_1,2a_1]$, $S_2 = S \setminus S_1$, $S_3 = S_2 - a_1$, and $T=[a_1,a_n]\setminus S$.
As $[a_1,2a_1] \not\subseteq S$, we have $|S_1| \leq a_1$.
Notice that every element of $S_3 \cap S$ is adjacent to $a_1$: for each $x \in S_3$, we have $a_1 + x \in S_2 \subseteq S$ and thus if $x \in S$ then $x$ is adjacent to $a_1$.
Hence, $|S\cap S_3| \leq \Delta$.
As $S_3 \subseteq [a_1 + 1, a_n - a_1]$, we find
\[ |S_3 \cap S| + |S_3 \cap T| = |S_3 \cap [a_1,a_n]| = |S_3| = |S_2| = |S| - |S_1| \geq n-a_1.\]

If $|S\cap S_3| = \Delta$, then $|S_3 \cap T| \geq n-\Delta - a_1$, and $a_1$ is already adjacent to $\Delta$ labels in $[a_1+1,a_n-a_1]$, so it cannot be adjacent to $a_n$.
We have
\[ (a_n - a_1 + 1) - n = |T| \geq |S_3 \cap T| \geq n - \Delta - a_1,\]
which implies $a_n \geq 2n - \Delta - 1$.
As $a_n$ is adjacent to at least $\delta\geq 1$ vertices and is not adjacent to $a_1$, it is adjacent to some vertex with label at least $a_{\delta+1}$, so we have
\[ \max L \geq a_n + a_{\delta + 1} \geq 2n - \Delta - 1 + a_1 + \delta, \]
so $\range(L) \geq 2n - (\Delta - \delta) - 1$, contradicting the assumption that $\range(L)=2n-(\Delta-\delta)-2$.

Otherwise, we have $|S\cap S_3| \leq \Delta - 1$, so $|S_3 \cap T| \geq n - \Delta - a_1 + 1$, and thus $a_n \geq 2n - \Delta$ following the same argument as before.
As $a_n$ is adjacent to at least $\delta \geq 1$ vertices, we have
\[ \range(L) \geq a_n + a_\delta - a_1 \geq 2n-\Delta + (\delta -1 ) = 2n - (\Delta-\delta) - 1,\]
again contradicting the assumption that $\range(L)=2n-(\Delta-\delta)-2$.

Hence, $[a_1,2a_1] \subseteq S$.
\end{proof}

\section{Spum of paths}\label{section: spum path}
In this section we improve the bounds on $\spum(P_n)$ given in \cite{singla2021some}, where $P_n$ denotes the path graph with $n$ vertices.
In particular, we prove the following result.
\begin{theorem}\label{theorem: spum path best bounds}
For $3 \leq n \leq 6$, we have $\spum(P_n) = 2n-3$, and for $n \geq 7$, we have
\begin{equation*}
    2n-2 \leq \spum(P_n) \leq \begin{cases} 2n+1 & \text{if } n \text{ is odd} \\ 2n-1 & \text{if } n \text{ is even}. \end{cases}
\end{equation*}
\end{theorem}

As $P_2 = K_2$ and the spum and integral spum of complete graphs are completely determined, we only consider $n \geq 3$.
It was shown in \cite{harary1990sum} that $\sigma(P_n)=1$, and in \cite{harary1994sum} that $\zeta(P_n)=0$, for all such $n$.

The previously known best bounds for $\spum(P_n)$ are
\begin{align}\label{eq: spum bound P_n}
    2n-3 \leq \spum(P_n) \leq \begin{cases} 2n+1 & \text{if } n \text{ is odd} \\ 2n+2 & \text{if } n \text{ is even}, \end{cases}
\end{align}
due to \cite{singla2021some}, where the upper bounds hold for $n \geq 9$ and the lower bound holds for $n \geq 3$.

We improve the upper bound in the following result, lowering the gap between upper and lower bounds from 5 to 2 in the even case.
\begin{theorem}\label{theorem: spum path upper bound}
For $n \geq 3$, we have $\spum(P_n) \leq \begin{cases} 2n+1 & \text{if } n \text{ is odd} \\ 2n-1 & \text{if } n \text{ is even}. \end{cases}$
\end{theorem}
\begin{proof}
For even $n \geq 4$, let $L=\{1,3,5,\dots,2n-3\}\cup\{2n-4,2n\}$.
We claim that $L$ induces $P_n$ with an additional isolated vertex as its sum graph, and as $\range(L) = 2n-1$, this proves the upper bound.
We claim the path formed is given by the sequence
\[ 2n-4, 1, 2n-5, 5, 2n-9, 9, 2n-13, \dots, 2n-7, 3, 2n-3,\]
where after $2n-4$ the sequence alternates between the integers equivalent to 1 modulo 4 from 1 to $2n-3$, inclusive, and the integers equivalent to 3 modulo 4 from $2n-5$ to 3, inclusive.

It suffices to show that these edges are the only edges in the induced sum graph.
If an edge is between two odd vertices, then they must sum to either $2n-4$ or $2n$.
The edges whose vertices sum to $2n-4$ are $(1,2n-5),(3,2n-7),\dots,(n-3,n-1)$, and the edges whose vertices sum to $2n$ are $(3,2n-3),(5,2n-5),\dots,(n-1,n+1)$.

No edge can be between two even vertices, as $2n=\max L$ must be isolated.
For an edge to be between an even and an odd vertex, the even vertex must be $2n-4$, and the only such edge is $(1,2n-4)$.
Hence, we find $L$ induces $P_n$.

For odd $n \geq 9$, the upper bound was shown in \cite[Theorem 7.1]{singla2021some}.
For $n\in\{3,5,7\}$, refer to the constructions in \cref{tab:spum P_n}, all of which show that $\spum(P_n) \leq 2n+1$ in this case.
\end{proof}

\begin{table}[htbp]
    \centering
    \setlength{\tabcolsep}{18pt}
    \begin{tabular}{lll}
        \hline
        $n$ & Lexicographically first optimal labeling & $\spum(P_n)$ \\
        \hline
        3 & \{1, 2, 3, 4\} & 3 \\
        4 & \{1, 2, 3, 4, 6\} & 5 \\
        5 & \{1, 2, 4, 5, 6, 8\} & 7 \\
        6 & \{1, 2, 4, 5, 7, 9, 10\} & 9 \\
        7 & \{1, 2, 4, 6, 7, 9, 12, 13\} & 12 \\
        8 & \{1, 2, 4, 6, 7, 9, 12, 15, 16\} & 15 \\
        9 & \{1, 2, 4, 5, 8, 12, 15, 17, 18, 20\} & 19 \\
        10 & \{1, 3, 5, 7, 9, 11, 13, 15, 16, 17, 20\} & 19 \\
        11 & \{1, 3, 5, 7, 9, 11, 13, 15, 16, 17, 19, 24\} & 23 \\
        12 & \{1, 3, 5, 7, 9, 11, 13, 15, 17, 19, 20, 21, 24\} & 23 \\
        13 & \{1, 3, 5, 7, 9, 11, 13, 15, 17, 19, 20, 21, 25, 28\} & 27 \\
        14 & \{1, 3, 5, 7, 9, 11, 13, 15, 17, 19, 21, 23, 24, 25, 28\} & 27 \\ 
        15 & \{1, 3, 5, 7, 9, 11, 13, 15, 17, 19, 21, 23, 24, 25, 27, 32\} & 31 \\
        \hline
    \end{tabular}
    \caption{Initial values of $\spum(P_n)$.}
    \label{tab:spum P_n}
\end{table}

\begin{remark}\label{remark: spum path upper bound tightness}
The upper bound for even $n$ is not sharp for $n\in\{4,6\}$, as $L=\{1,2,3,4,6\}$ induces $P_4$ with $\range(L)=5<2n-1=7$, and $L=\{1,2,4,5,7,9,10\}$ induces $P_6$ with $\range(L)=9<2n-1=11$, but is sharp for even $n$ between 8 and 14, inclusive.
Similarly, the upper bound for odd $n$ is not sharp for $n\in\{3,5,7\}$, but is sharp for odd $n$ between 9 and 15, inclusive.

Recall that $\sigma(P_n) = 1$ for $n \geq 3$.
Let $x = 2n-3$.
If $L$ is a labeling that induces $P_n$ with range $x$.
Sorting the $n$ vertices of our path graph $a_1 < \cdots < a_n$, we find $a_n$ must be adjacent to some vertex $a_i$ for $i < n$, meaning $a_n + a_i \in L$, so $a_1 + x = \max L \geq a_n + a_i \geq a_n + a_1 \geq 2a_1 + n - 1$, which implies $1 \leq \min L = a_1 \leq x - n + 1$.
Hence, $\spum(P_n) = x$ if and only if there exists a set $L \subset [1, 2x-n+1]$ with range $x$ that induces $P_n$.
Otherwise, we increase $x$ by one and repeat.
Using an exhaustive computer search, we can use this to compute $\spum(P_n)$ for $3 \leq n \leq 15$.
The results are in \cref{tab:spum P_n}.
A similar table was provided in \cite{singla2021some}, though they erroneously list $\spum(P_n)=2n+2$ for $n=10,12,14$, where our construction demonstrates $\spum(P_n) \leq 2n-1$ for these values; we also provide more values for $n$ in our table.
\end{remark}

Our construction of a $2n-1$ bound for $\spum(P_n)$ when $n$ is even falsifies a previous conjecture on $\spum(P_n)$ by \cite[Conjecture 7.1]{singla2021some}, but we provide the following updated conjecture.

\begin{conjecture}\label{conjecture: spum P_n}
For $n \geq 8$, we have $\spum(P_n) = \begin{cases} 2n+1 & \text{if } n \text{ is odd} \\ 2n-1 & \text{if } n \text{ is even}. \end{cases}$
\end{conjecture}

\cref{tab:spum P_n} verifies this conjecture for $8 \leq n \leq 15$.

For $3 \leq n \leq 6$, we observe $\spum(P_n)=2n-3$ achieves its lower bound from \cref{eq: spum bound P_n} following \cref{theorem: spum path best bounds}, but for $n \geq 7$, this lower bound inequality appears to be strict; \cref{tab:spum P_n} shows this inequality is strict for $7 \leq n \leq 15$.
The following result shows that this is indeed true for all $n \geq 7$, as the lower bound can be improved.

\begin{theorem}\label{theorem: spum path lower bound}
For $n \geq 7$, we have $\spum(P_n) \geq 2n-2$.
\end{theorem}
\begin{proof}
From \cite{singla2021some} we have $\spum(P_n) \geq 2n-3$, so assume for the sake of contradiction that $\spum(P_n) = 2n-3$.
Suppose $L$ is a labeling with $\range(L)=2n-3$ that induces $P_n$ with one additional isolated vertex.
Suppose the set of labels of $P_n$ is $S=\{a_1,\dots,a_n\}$, where $a_1 < \cdots < a_n$, and suppose the isolated vertex is labeled $b$, so that $L=\{a_1,\dots,a_n\} \cup \{b\}$.
From \cref{lemma: spum lower bound equality a1-2a1}, $\spum(P_n)=2n-3$ implies $[a_1,2a_1]\subseteq S$.

As $a_n$ has at least one neighbor, $\max L \geq a_n + a_1 \geq 2a_1 + n - 1$, so $2n-3 = \range(L) \geq a_1 + n - 1$, and thus $a_1 \leq n-2$.
We now eliminate the cases where $a_1 = n-2$ or $a_1 = n-3$ so that we may assume $a_1 \leq n-4$.

If $a_1 = n-2$, then $b = 3n-5$.
We have $a_n \geq 2n-3$, and in order to be adjacent to some vertex this means $a_n = 2n-3$ and $a_n$ is (only) adjacent to $a_1$.
Hence $S=[n-2,2n-3]$, and we have a path $a_n,a_1,a_2,a_{n-1}$.
But $a_{n-1}=2n-4$ is not adjacent to any other vertices, so as $n \geq 7$, we find $L$ does not induce $P_n$ as assumed.
Thus $a_1\neq n-2$.

If $a_1 = n-3$, then $b=3n-6$, and either $a_n=2n-4$ or $a_n = 2n-3$.

Case 1: $a_n = 2n-4$.
This yields $S=[n-3,2n-4]$ and we have a path $a_n,a_2,a_1,a_3,a_{n-1}$, where $a_{n-1}=2n-5$.
However, $a_{n-1}$ is not adjacent to any other vertices, and as this path only has 5 vertices and $n \geq 7$, we find $L$ does not induce $P_n$ as assumed.

Case 2: $a_n = 2n-3$.
This yields $S=[n-3,2n-3]\setminus\{k\}$ for some $n-3<k<2n-3$.
In other words, it is only missing one integer in the interior of $[n-3,2n-3]$.
We notice that $a_1+a_n=b$ so $a_1$ is adjacent to $a_n$.
We also notice that $a_1 + n = a_n$, $a_1 + (n-1) = 2n-4$, and $a_1 + (n-2) = 2n-5$.
For $n \geq 7$, we have $n-2, n-1, n, 2n-5$, and $2n-4$ are all distinct labels in the interior of $[n-3,2n-3]$, and at most one of them is missing from $S$, which means that for at least two of the three given equations, all terms are in $S$ and thus $a_1$ is adjacent to at least two vertices with labels in $[n-2,n]$.
As $a_1$ is also adjacent to $a_n$, this yields $a_1$ has degree at least 3, which is forbidden.

We may now assume $a_1 \leq n-4$.
As $[a_1,2a_1]\subset S$ and $a_1+1 \leq n-3$, we have at least 3 elements in $S$ strictly larger than $2a_1$.
Any element $x\in S$ contained in $[2a_1+1,3a_1]$ would yield $x-a_1$ being adjacent to $a_1$, so as $\deg(v_1) \leq 2$, we can have at most 2 elements of $S$ in $[2a_1+1,3a_1]$.
If there are $x \leq 2$ such elements, there are at least $3-x \geq 1$ elements of $S$ strictly greater than $3a_1$, and the largest such element is at least $3a_1 + 3-x$.
If $x < 2$, this element is at least $3a_1+2$, and as its vertex has degree at least 1, we have $\max L \geq 3a_1 + 2 + a_1 = 4a_1 + 2$.
Otherwise $x=2$, which implies $v_1$ is already adjacent to two vertices whose labels are in $[a_1+1,2a_1]$, so $\max S \geq 3a_1 + 1$ and its vertex has degree at least 1, but is not adjacent to $v_1$, so we still have $\max L \geq 3a_1 + 1 + a_1 + 1 = 4a_1 + 2$.
This means $\spum(P_n)=2n-3 \geq 3a_1 + 2$, so $a_1 \leq \frac{2n-5}{3}$.

Similarly, the same reasoning as in the proof of \cite[Claim 2]{singla2021some} with only some very minor modifications implies $a_1 \geq n-7$.
The inequalities $n-7 \leq a_1 \leq \frac{2n-5}{3}$ yield no integer solution for $a_1$ when $n \geq 17$, so for $n \geq 17$ we have $\spum(P_n) \neq 2n-3$.

\cref{tab:spum P_n} gives that $\spum(P_n) \geq 2n-2$ for $7 \leq n \leq 15$.
The final remaining case of $n=16$ can be easily verified via computer search.
We must have $a_1=9$ and $[9,18]\subset S$ with $b=9+29=38$, which leaves 6 more labels for $S$ that must be within $[19,37]$, so only $\binom{19}{6}=27132$ labelings must be checked.
\end{proof}

Together, these results prove \cref{theorem: spum path best bounds}.

\section{Spum of cycles}\label{section: spum cycle}
In this section we show that $\spum(C_n)=2n-1$ for $n \geq 4$, where $C_n$ is the cycle graph on $n$ vertices.
This corrects for a logical flaw in the proof by Singla, Tiwari, and Tripathi \cite{singla2021some} that $\spum(C_n)=2n-1$ for $n \geq 13$, as well as extending it to include $4 \leq n \leq 12$.

The spum and integral spum were determined exactly for complete graphs in \cite{singla2021some}, so as the cycle graphs are complete for $n \leq 3$, we only consider cycle graphs $C_n$ for $n \geq 4$.

We know from \cite{harary1990sum} that the sum number $\sigma(C_n)=2$, except for $n=4$ where $\sigma(C_4)=3$.
Likewise, from Sharary \cite{sharary1996integral} the integral sum number $\zeta(C_n)=0$, except for $n=4$ where $\zeta(C_4)=3$, from Xu \cite{xu1999integral}.

\begin{remark}\label{remark: singla error spum cycle proof}
The authors of \cite[Theorem 6.1]{singla2021some} show that $\spum(C_n) \leq 2n-1$ for all $n \geq 4$; however, as $\sigma(C_n)$ is different for $n=4$, their construction does not hold for $n=4$.
We provide a valid construction in \cref{theorem: spum C_n} for this case, and thus the result still holds, despite the original proof not accounting for $n=4$.

A lower bound of $\spum(C_n) \geq 2n-2$ for all $n \geq 4$ follows from \cite[Theorem 2.1]{singla2021some}, so $2n-2 \leq \spum(C_n) \leq 2n-1$.

Then in \cite[Theorem 6.2]{singla2021some}, it is claimed that $\spum(C_n)=2n-1$ for all $n \geq 13$.
However, in this proof, in particular the proof of Claim 1, they argue that a vertex with label $2a_1$, where $a_1 = \min L$ is the minimum label, is adjacent to a vertex with label greater than $a_1$.
They then claim that this implies the maximum label $a_n$ associated to one of the $n$ vertices of $C_n$ is greater than $3a_1$.
Unfortunately, while it can be concluded that a vertex with label greater than $3a_1$ exists in $L$, this vertex does not necessarily belong to $C_n$, as it can be one of the two isolated vertices.
As a simple example of this, consider the labeling $L=[3,6]\cup[8,10]$, which induces $C_4$ with the four labels in $[3,6]$ constituting $C_4$, while the three labels in $[8,10]$ are additional isolated vertices.
In this case $a_1 = 3$, and $2a_1=6$ is indeed adjacent to a label greater than $a_1=3$, namely 4.
Hence $\max L = 10 > 9 = 3a_1$, but their statement that $a_n > 3a_1$ is false, as $a_n=6$ while $3a_1 = 9$.

In \cref{theorem: spum C_n}, we prove this result, and extend it to also address the $4 \leq n \leq 12$ cases, thus including all $n \geq 4$.
So while the original proof may not be correct, the result of \cite[Theorem 6.2]{singla2021some} is still true.
Many of the techniques used in this proof are inspired by the original argument from \cite{singla2021some}.
\end{remark}

\begin{theorem}\label{theorem: spum C_n}
For $n \geq 3$, we have $\spum(C_n) = \begin{cases} 6 & n = 3 \\ 2n-1 & n \geq 4. \end{cases}$
\end{theorem}
\begin{proof}
If $n=3$, we recover $C_3=K_3$, whose spum was found in \cite{singla2021some} to be 6.

We first separately address the case $n=4$.
As $\sigma(C_4)=3$, we have 7 total vertices.
Notice that $L=[3,6]\cup[8,10]$ induces a sum graph consisting of $C_4$ and three isolated vertices, so $\spum(C_4) \leq \range(L)=7$.

Assume for the sake of contradiction that $\spum(C_4)\neq 7$, so we have $\spum(C_4)=6$, meaning the labels have to be a consecutive block of 7 positive integers.
Let the smallest of these seven integers be $a$.
For $a \leq 3$, we have that $a$ has degree strictly larger than 2, so this cannot yield $C_4$.
For $a \geq 4$, the only edges in the induced sum graph are incident to $a$, because for any two distinct vertices $b$ and $c$ that are both not $a$, we have $b+c \geq 2a+3 \geq a+7$, while $\max L = a+6$.
This cannot yield $C_4$.
Hence we have $\spum(C_4)=7$, as desired.

From \cite[Remark 6.1]{singla2021some}, we know $2n-2 \leq \spum(C_n) \leq 2n-1$ for all $n \geq 5$.

Assume for the sake of contradiction that $\spum(C_n)=2n-2$.
Suppose the vertices of $C_n$ are labeled $a_1<\cdots<a_n$, and the two isolated vertices are labeled $b<c$ in a labeling $L$ that achieves $\range(L) = 2n-2$.
Notice that in order for $a_n$ to be adjacent to two other labels $a_i$ and $a_j$ for $i < j < n$, we must have $b=a_i+a_n$ and $c=a_j+a_n$, so the total ordering of our $n+2$ labels is $a_1 <\cdots<a_n<b<c$.
Let $S=\{a_i\mid 1 \leq i \leq n\}$.
From \cref{lemma: spum lower bound equality a1-2a1}, we know $\spum(C_n)=2n-2$ implies $[a_1,2a_1] \subset S$.
Notice that $a_1 \leq n-2$, as $a_n \geq a_1 + n-1$, so
\[ \spum(C_n) \geq (a_1+n-1+a_1+1)-a_1 = a_1 + n, \]
and thus $a_1+n \leq 2n-2$ yields $a_1 \leq n-2$.

We first eliminate the possibility that $a_1 = n-2$.
If this were so, then $\max L = c = 3n-4$.
We have $a_n \geq 2n-3$, and as $a_n$ is adjacent to two labels in $S$, this yields
\[ \max L \geq (2n-3)+(n-1)=3n-4.\]
As we need this inequality to be sharp, we require $a_n = 2n-3$ and $a_n$ being adjacent to $a_1$ and $a_2$.
With $a_n=2n-3$, this means $S = [n-2,2n-3]$.
We must have $a_n + a_1 = 3n-5 \in L\setminus S$, so $b=3n-5$.
However, we find that $a_1 + a_2 = 2n-3 = a_n$, so $a_1$ and $a_2$ are adjacent.
Thus $a_1$, $a_2$, and $a_n$ form a triangle, so this does not yield $C_n$ for $n \geq 5$.

Now we eliminate the possibility that $a_1 = n-3$.
If this were so, then $\max L = c = 3n-5$.
We have $a_n \geq 2n-4$, so $\max L \geq (2n-4)+(n-2)=3n-6$.
As $c = 3n-5$, we have a slackness of 1, which affords only a small number of cases:
\begin{enumerate}[label=\arabic*)]
    \item $a_n = 2n-4$ and $a_n$ is adjacent to $a_1$ and $a_2$.
    
    As $a_n = 2n-4$, this requires $S = [n-3,2n-4]$.
    We have $a_1 + a_n = 3n-7 \in L\setminus S$, so $b=3n-7$.
    We also have $a_2 + a_n = (n-2)+(2n-4) = 3n-6\in L\setminus S$, so $c=3n-6$, contradicting the fact that $c=3n-5$.
    \item $a_n = 2n-4$ and $a_n$ is adjacent to $a_1$ and $a_3$.
    
    Notice that $a_1 + a_3 = (n-3) + (n-1) = 2n-4 = a_n$, so we have a triangle between $a_1$, $a_3$, and $a_n$, contradicting the fact that our induced sumgraph is $C_n$ for $n \geq 5$.
    \item $a_n = 2n-4$ and $a_n$ is adjacent to $a_2$ and $a_3$.
    
    We have $a_2 + a_n = 3n-6 = b$, while $a_3 + a_n = 3n-5 = c$.
    Then $a_{n-1}=2n-5$ is adjacent to $a_3 = n-1$ as $a_3 + a_{n-1} = b$.
    So $a_3$ is adjacent to $a_n$ and $a_{n-1}$.
    However, notice that $a_1 + a_3 = (n-3) + (n-1) = 2n-4 = a_n$, so $a_3$ is also adjacent to $a_1$, giving $a_3$ degree 3, contradicting the fact that each vertex in $C_n$ has degree 2.
    \item $a_n = 2n-3$ and $a_n$ is adjacent to $a_1$ and $a_2$.
    
    We have $a_2 \geq n-2$, but $a_2 + a_n \geq 3n-5 = c$, so we need equality and thus $a_2 = n-2$.
    We find $a_1 + a_n = 3n-6 \in L \setminus S$, so $b=3n-6$.
    With $a_n = 2n-3$, we have $S = [n-3,2n-3]\setminus\{k\}$ for some integer $k$ where $n-1\leq k\leq 2n-4$.
    We cannot have an edge between $a_1$ and $a_2$, or else $a_1$, $a_2$, and $a_n$ form a triangle.
    So $a_1 + a_2 = 2n-5 \not\in S$, which means $k=2n-5$.
    
    Then $a_1 + a_3 = 2n-4 \in S$, so $a_1$ and $a_3$ are adjacent.
    If $n \geq 6$, then $2n-5 \geq n+1$, so $a_4 = n$ and $a_1 + a_4 = 2n-3 = a_n$, so $a_1$ is adjacent to $a_3$, $a_4$, and $a_n$, contradicting the fact that it must have degree 2.
    
    Otherwise $n=5$, and we have $S=\{2,3,4,6,7\}$ with $b=9$ and $c=10$.
    We observe 3 is adjacent to 4, 6, and 7, contradicting the fact that it must have degree 2.
\end{enumerate}

So we can now assume $a_1 \leq n-4$.
As $[a_1,2a_1] \subset S$, which accounts for $a_1 + 1 \leq n-3$ of the elements of $S$, we have at least 3 elements in $S$ that are strictly larger than $2a_1$.
Any such element $x$ contained in $[2a_1+1,3a_1]$ would cause $x-a_1$ to be adjacent to $a_1$, so as $\deg(v_1)=2$, we can have at most 2 elements of $S$ in $[2a_1+1,3a_1]$.
If there are $x\leq 2$ such elements, there are at least $3-x \geq 1$ elements of $S$ greater than $3a_1$.
The largest such element is at least $3a_1+3-x$, and if $x< 2$ then this element is at least $3a_1 + 2$, and its vertex has degree 2, so $\max L \geq 3a_1 + 2 + a_1 + 1 = 4a_1 + 3$.
Otherwise $x=2$, which implies $v_1$ is already adjacent to two vertices whose labels are in $[a_1+1,2a_1]$, so while the element is at least $3a_1+1$, its vertex has degree 2 and is not adjacent to $v_1$, so we still have $\max L \geq 3a_1 + 1 + a_1 + 2 = 4a_1 + 3$.
This means $\spum(C_n) = 2n-2 \geq 3a_1 + 3$, so $a_1 \leq \frac{2n-5}{3}$.

On the other hand, we have $a_1 \geq n-7$ from \cite[Claim 2]{singla2021some}.

We now analyze what occurs if $a_1 \leq n-6$.
As $a_1$ is adjacent to all elements of
\[ \left(S\cap [2a_1+1,3a_1]\right)-a_1 \subseteq [a_1+1,2a_1], \]
we have $|S\cap [2a_1+1,3a_1]| \leq 2$, and similarly, as $2a_1$ is adjacent to 2 labels, we have $|S\cap [3a_1,4a_1-1]| \leq 2$.
So
\[ |S\cap [a_1,4a_1-1]| \leq a_1 + 5 \leq n-1,\]
and thus we have $a_n \geq 4a_1 + n - a_1 - 6 = 3a_1 + n - 6$.
Let $x=|S\cap [2a_1+1,3a_1]|$; if $x=2$, then $a_1$ cannot be adjacent to $a_n$, meaning $\max L \geq 4a_1 + n - 4$.
Otherwise $x < 2$, so we improve our original bound $|S\cap [a_1,4a_1-1]| \leq a_1 + 4$, yielding $a_n \geq 4a_1 + n - a_1 - 5$, and so $\max L \geq 4a_1 + n - 4$.
In either case, we have $a_1 + 2n-2 =\max L \geq 4a_1 + n - 4$, which yields $a_1 \leq \frac{n+2}{3}$.

Notice that for $n \geq 11$, we have $a_1 \leq \left\lfloor\frac{2n-5}{3}\right\rfloor \leq n-6$, and thus $n-7 \leq a_1 \leq \frac{n+2}{3}$, which admits no integer solutions for $a_1$ when $n \leq 11$.
So we have proven the result for all $n \geq 12$.

The only remaining cases are $5 \leq n \leq 11$, where we know $\max\{1,n-7\} \leq a_1 \leq \left\lfloor\frac{2n-5}{3}\right\rfloor$.

We now show that $a_1 \geq 2$ by showing that if $a_1=1$, the degree of $a_1$ is at least 3.
If $a_1=1$, then $c=2n-1$ and $a_1$ could be adjacent to $x$, for all integers $2 \leq x \leq 2n-2$.
There are $2n-3$ such integers $x$.
The only way for $x$ to not be adjacent to $a_1$ is if $x\not\in L$ or $x+1 \not\in L$.
As $b$ and the $n-1$ other labels in $S$ are contained in $[2,2n-2]$, there are only $2n-3-n=n-3$ values in this interval not contained in $L$.
Each such missing value $y$ can prevent $a_1$ from being adjacent to $y$ and $y-1$, so these $n-3$ missing values obstruct at most $2n-6$ potential edges out of the $2n-3$, meaning that $a_1$ has at least 3 incident edges.

For $n=5$, our original $a_1 \leq n-4$ bound gives $a_1 \leq 1$, but we also need $a_1 \geq 2$, which yields no solution.
For $n=11$, we have $a_1 \leq 5 = n-6$ which requires $a_1 \leq \frac{n+2}{3}$, so we can eliminate the possibility of $a_1 = 5$.
Hence, we have reduced our problem to $6 \leq n \leq 11$, with the following bounds on $a_1$ for each remaining value of $n$, derived from $\max\{2,n-7\} \leq a_1 \leq \left\lfloor\frac{2n-5}{3}\right\rfloor$:
\begin{itemize}
    \item If $n=6$, then $2 \leq a_1 \leq 2$.
    \item If $n=7$, then $2 \leq a_1 \leq 3$.
    \item If $n=8$, then $2 \leq a_1 \leq 3$.
    \item If $n=9$, then $2 \leq a_1 \leq 4$.
    \item If $n=10$, then $3 \leq a_1 \leq 5$.
    \item If $n=11$, then $4 \leq a_1 \leq 4$.
\end{itemize}

These remaining cases can be easily ruled out by computer search.
We have bounded $a_1 = \min L$, and as we assume $\range(L) = 2n-2$, we have $\max L = a_1 + 2n-2$.
Between these two values we must pick $n$ more integer values to form $L$, and one can simply check that the induced sum graphs for all such subsets do not yield $C_n$.

If one wanted to do these remaining cases by hand, the following is an example for the case $n=6$.
The other cases would follow similarly, though they would be tedious.

For $n=6$, we have $2,3,4\in S$, and $c=12$.
Among 5, 6, and 7, we can have at most 2 elements; otherwise, 2 is adjacent to 3, 4, and 5.
This yields $a_n \geq 8$, but also notice that $\max L = 12 \geq a_n + a_2 = a_2 + 3$, so $a_n \leq 9$.
If $a_n = 8$, then 8 and 4 are adjacent, and 8 is also adjacent to either 2 or 3.
However, 8 cannot be adjacent to 2, as then $b=10$, but then $a_1=2$ and $b=10$ would be adjacent.
So 8 is adjacent to 3 and 4; additionally, 7 cannot be in $S$, as otherwise 3, 4, and 8 would form a triangle.
So the remaining two elements must be 5 and 6, but this yields 2 being adjacent to 3 and 4, forming a 4-cycle with 2, 3, 4, and 8, while we want a 6-cycle.
Otherwise, we have $a_n=9$, which must be adjacent to 2 and 3, so $b=11$.
As 2 is adjacent to some element in $\{3,4,5\}$ for each element in $S\cap \{5,6,7\}$, we have $|S \cap \{5,6,7\}| \leq 1$, forcing $a_5 = 8$.
This gives a path 2--9--3--8--4, and the only element in $\{5,6,7\}$ which when added to $S$ would create an edges between itself and both 2 and 4 is 7, which would cause an additional edge between 3 and 4, thus not creating $C_6$.
\end{proof}

\section{Integral spum of cycles}\label{section: ispum cycle}
The spum of cycles has now been determined as $\spum(C_n)=2n-1$.
It was previously bounded within $[2n-2,2n-1]$; on the other hand, the best previously known bounds for $\ispum(C_n)$ are significantly worse, namely
\begin{align}\label{eq: ispum cycle original bound}
    2n-5 \leq \ispum(C_n) \leq \begin{cases} 17(n-9) & \text{if } n \text{ is odd} \\ \frac{3}{2}(3n-14) & \text{if } n \text{ is even},\end{cases}
\end{align}
where this upper bound, due to Melnikov and Pyatkin \cite{melnikov2002regular}, only holds for $n \geq 10$.
The lower bound is due to Singla, Tiwari, and Tripathi \cite{singla2021some}.

Recall that $\zeta(C_n)=0$ for $n \geq 5$.
Let $x = 2n-5$.
If $L$ is a labeling of $C_n$ with range $x$, then it must contain at least one label of each sign, so $1-x \leq \min L \leq -1$ and $1 \leq \max L \leq x-1$.
Hence, $\ispum(C_n) = x$ if and only if there exists a set $L \subset [1-x, 1+x]$ with range $x$ that induces $C_n$.
Otherwise, we increase $x$ by one and repeat.
Using an exhaustive search on a computer, we can use this to compute $\ispum(C_n)$ for $5 \leq n \leq 14$.
The results are in \cref{tab:ispum C_n}.

For $n=4$, we have $\zeta(C_4)=\sigma(C_4)=3$, so $\ispum(C_4) \geq 6$.
As 0 cannot be a label, if the labels are all of the same sign, then $\range(L) \geq \spum(C_4)=7$.
If there exist positive and negative labels, then $0 \in [\min L, \max L] \setminus L$, so $\range(L) \geq 7$.
A construction with $\range(L)=7$ is possible, so $\ispum(C_4)=7$.

\begin{table}[htbp]
    \centering
    \setlength{\tabcolsep}{12pt}
    \begin{tabular}{lll}
        \hline
        $n$ & Lexicographically first optimal labeling & $\ispum(C_n)$ \\
        \hline
        4 & $\{-10, -9, -8, -6, 5, -4, -3\}$ & 7 \\
        5 & $\{-3, -2, -1, 1, 2\}$ & 5 \\
        6 & $\{-5, -3, -2, -1, 2, 3\}$ & 8 \\
        7 & $\{-7, -5, -4, -3, 1, 2, 4\}$ & 11 \\
        8 & $\{-11, -10, -8, -7, -3, -1, 1, 3\}$ & 14 \\
        9 & $\{-9, -8, -7, -4, -2, 1, 2, 4, 8\}$ & 17 \\
        10 & $\{-13, -12, -10, -9, -4, -3, -2, 2, 3, 4\}$ & 17 \\
        11 & $\{-16, -15, -12, -11, -5, -4, -3, -2, 3, 4, 5\}$ & 21 \\
        12 & $\{-19, -18, -17, -14, -13, -6, -5, -3, 3, 4, 5, 6\}$ & 25 \\
        13 & $\{-20, -19, -15, -14, -7, -6, -5, -4, -3, 3, 4, 5, 6\}$ & 26 \\
        14 & $\{-26, -25, -22, -21, -19, -14, -13, -12, -6, -5, -4, 3, 4, 5\}$ & 31 \\
        \hline
    \end{tabular}
    \caption{Initial values of $\ispum(C_n)$.}
    \label{tab:ispum C_n}
\end{table}

We improve the upper bounds in the following result.
\begin{theorem}\label{theorem: ispum cycle upper bound}
For $n \geq 12$, we have $\ispum(C_n) \leq \begin{cases} 8(n-9) & \text{if } n \text{ is odd} \\ \frac{3}{2}(3n-14) & \text{if } n \text{ is even}.\end{cases}$
\end{theorem}
\begin{proof}
For the even case, the result is unchanged from \cref{eq: ispum cycle original bound}.

For the odd case, we reference \cref{tab:ispum C_n} to find that the bound holds for $n=13$.
We now assume $n\geq15$.
We construct a labeling $L$ for $C_{2k+9}$, where $k \geq 3$.

Let $L$ be the set
\begin{align*}
    L = [-8k,-7k+1]\cup[4k,5k]\cup\{-3k,-3k+1\}\cup\{-5k,-k-1\}\cup\{7k-1,8k\}.
\end{align*}
This induces a cycle given by the following list of cyclically adjacent vertices:
\begin{align*}
    -8k,5k,-8k+1,5k-1,\dots,4k,-7k+1,-k-1,8k,-3k,-5k,-3k+1,7k-1,-8k.
\end{align*}
The initial sequence is an interleaved sequence of $-8k$ increasing to $-7k+1$ and $5k$ decreasing to $4k$.
One can easily confirm that the edges in this cycle exist, and that $L$ consists of $2k+9$ distinct labels.
It is slightly more tedious, but straightforward, to verify that no other edges exist in the induced sum graph, as long as $k \geq 3$.

Hence, $\ispum(C_n) \leq \range(L) = 16k = 8(n-9)$.
\end{proof}

\section{Spum and integral spum of matchings}\label{section: spum ispum nK2}
The spum and integral spum of matchings $nK_2$, the disjoint union of $n$ copies of $K_2$, have not been studied previously.
However, Harary \cite{harary1990sum} showed that $\zeta(nK_2)=0$ for all positive integers $n$, so the study of these spum values is tractable.
Before we address the integral spum of $nK_2$, however, we first have the following result for spum.
\begin{theorem}\label{theorem: spum nK2}
For all positive integers $n$, we have $\sigma(nK_2)=1$ and $\spum(nK_2)=4n-2$.
\end{theorem}
\begin{proof}
As no vertex in $nK_2$ is isolated, clearly we need $\sigma(nK_2) \geq 1$, and $\sigma(nK_2) \leq 1$ due to the labeling $L=[2n-1,4n-2]\cup \{6n-3\}$, where the $2n$ labels in $[2n-1,4n-2]$ are matched up in pairs that sum to $6n-3$, and as $(2n-1)+2n = 4n-1 > 4n-2$, no other edges exist.
Thus $L$ induces $nK_2$ with one additional isolated vertex, so we find $\sigma(nK_2)=1$, and as $\range(L)=(6n-3)-(2n-1)=4n-2$, we have $\spum(nK_2) \leq 4n-2$.
On the other hand, the general lower bound $\spum(G) \geq 2|V|-(\Delta-\delta)-2$ given by \cite[Theorem 2.1]{singla2021some} gives $\spum(nK_2) \geq 4n-2$, as $nK_2$ has $2n$ vertices and all vertices have degree 1.
Thus, $\spum(nK_2) = 4n-2$.
\end{proof}
We now determine $\ispum(nK_2)$ for all values of $n$.
\begin{theorem}\label{theorem: ispum nK2}
For all positive integers $n$, we have $\ispum(nK_2)=\begin{cases} 4 & n=2 \\ 4n-3 & n \neq 2.\end{cases}$
\end{theorem}
\begin{proof}
When $n=1$ we recover $K_2$, whose integral spum can easily be seen is 1 using the labeling $L=\{0,1\}$, as first observed by \cite{singla2021some}.
For $n \geq 2$, we lower bound $\ispum(nK_2) \geq 4n-4$ using the general lower bound from \cite[Theorem 2.2]{singla2021some}.
Equality is reached for $n=2$ using the labeling $L=\{-2,-1,1,2\}$.

We now address $n\geq 3$.
We first demonstrate that $\ispum(nK_2) \leq 4n-3$, using the labeling $L=\{-1,1,3,\dots,4n-5\}\cup\{4n-4\}$.
The only edges between two odd labels must sum to the unique even label $4n-4$, so the $2n-2$ labels between 1 and $4n-5$, inclusive, are perfectly matched by those that sum to $4n-4$.
In order to have an edge incident to an even label, i.e., to $4n-4=\max L$, the other incident vertex must be negative, i.e., is $-1$.
So we find $L$ induces $nK_2$, and thus $\ispum(nK_2) \leq \range(L)=4n-3$.

We now closely follow the proof of \cite[Theorem 2.2]{singla2021some} in order to eliminate the equality case of the lower bound, i.e., to show that $\ispum(nK_2) \neq 4n-4$, and thus $\ispum(nK_2)=4n-3$.
Assume for the sake of contradiction that we have a labeling $L$ that achieves $\range(L)=\ispum(nK_2)=4n-4$; enumerate $L=\{a_1,\dots,a_{2n}\}$ in increasing order $a_1 < \cdots < a_{2n}$.
As $0 \not\in L$ and $L$ must contain integers of both signs, define $r$ as the unique index such that $a_r < 0 < a_{r+1}$.
Using $-L$ instead of $L$ if necessary, we may assume $a_{r+1} \leq |a_r| = -a_r$.

Define $S_1=\{a_1,\dots,a_r\}$, $S_2=\{a_{r+1},\dots,a_{2n}\}$, $S_3=S_1+a_{r+1}$, and $S_4 = S_2 - a_{r+1}$.
First, notice that for any $x \in S_1 \cap S_3$, we have $x-a_{r+1}$ is adjacent to $a_{r+1}$.
Similarly, for any $x \in (S_2 \cap S_4) \setminus \{a_{r+1}\}$, we have $x$ is adjacent to $a_{r+1}$.
As the degree of $a_{r+1}$ is 1, we have
\[ |(S_1 \cap S_3) \cup ((S_2 \cap S_4) \setminus \{a_{r+1}\})| = 1,\]
so
\begin{equation}\label{eq: nK2 sharp 1}
    |(S_1 \cap S_3) \cup (S_2 \cap S_4)| \leq 2,
\end{equation}
with equality if and only if $a_{r+1} \in S_2 \cap S_4$, or equivalently $2a_{r+1} \in S_2$.
All four sets $S_1,S_2,S_3$, and $S_4$ are contained in $[a_1,a_n]$, and $S_i \cap S_j = \emptyset$ for all $1 \leq i < j \leq 4$ except possibly $(i,j) \in \{(1,3),(2,4),(3,4)\}$.
As $\max S_3 = a_r + a_{r+1} \leq 0$ and $\min S_4 = 0$, we have
\begin{equation}\label{eq: nK2 sharp 2}
    |S_3 \cap S_4| \leq 1,
\end{equation}
with equality if and only if $a_r = -a_{r+1}$.
Combining \cref{eq: nK2 sharp 1} and \cref{eq: nK2 sharp 2} yields
\begin{equation*}\label{eq: nK2 sharp 3}
    \range(L) \geq \left|\bigcup_{i=1}^4 S_i\right| -1 \geq \sum_{i=1}^4 |S_i| - 3 - 1 = 4n-4,
\end{equation*}
with equality if and only if the following three conditions hold:
first,
\[ [\min L, \max L]=[a_1,a_{2n}]=\bigcup_{i=1}^4 S_i; \]
second, \cref{eq: nK2 sharp 1} must be equality, which occurs if and only if $2a_{r+1} \in S_2$;
and third, \cref{eq: nK2 sharp 2} must be equality, which occurs if and only if $a_r = -a_{r+1}$.

The only set $S_i$ that can contain any integers in $[1,a_{r+1}-1]$ is $S_4$, so $[1,a_{r+1}-1]\subseteq S_4$, and thus $S_2 \supseteq [a_{r+1},2a_{r+1}]$, which uses our previous observation that $2a_{r+1}\in S_2$.
Similarly, the only set $S_i$ that can contain any integers in $[-a_{r+1}+1,-1]$ is $S_3$, so $[-a_{r+1}+1,-1] \subseteq S_3$, and thus $[-2a_{r+1}+1,-a_{r+1}-1] \subseteq S_1$.
We also have $a_r=-a_{r+1}\in S_1$.
Notice that
\begin{equation}\label{eq: nK2 helper 1}
    |(S_1\cap S_3)\cup ((S_2 \cap S_4)+a_{r+1})|=2,
\end{equation}
as the two sets are disjoint and $|S_1\cap S_3| + |S_2\cap S_4|=2$.
As $a_r=-a_{r+1}$ is adjacent to all elements in $((S_1\cap S_3)\setminus\{-a_{r+1}\})\cup ((S_2 \cap S_4)+a_{r+1})$, but the degree of $a_r$ is 1, \cref{eq: nK2 helper 1} means $a_r=-a_{r+1}\in S_1 \cap S_3$, so $-2a_{r+1}\in S_1$.
This yields $[-2a_{r+1},-a_{r+1}] \subseteq S_1$.
We have shown the disjoint sets $S_1\subset L$ and $S_2\subset L$ each contain $a_{r+1}+1$ distinct labels, so $|L|=2n \geq 2a_{r+1}+2$, and thus $a_{r+1} \leq n-1$.
Furthermore, we have found $-a_{r+1}\in S_1 \cap S_3$ and $a_{r+1}\in S_2\cap S_4$, and these two intersections collectively have 2 elements, so we conclude $S_1 \cap S_3 = \{-a_{r+1}\}$ and $S_2 \cap S_4 = \{a_{r+1}\}$.
Thus $a_{r+1}$ is adjacent to $-2a_{r+1}$, and $-a_{r+1}=a_r$ is adjacent to $2a_{r+1}$.
Any element $x \in L \cap [2a_{r+1}+1,3a_{r+1}]$ yields $x-a_{r+1}$ is adjacent to $a_{r+1}$, but as $a_{r+1}$ is already matched, this cannot occur, so $L$ is disjoint from $[2a_{r+1}+1,3a_{r+1}]$.
If $1 < a_{r+1} < n-1$, so that there are labels in $L\setminus([-2a_{r+1},-a_{r+1}]\cup[a_{r+1},2a_{r+1}])$, then without loss of generality suppose this label is positive (as we can use $-L$ instead).
Then $\max L > 3a_{r+1}$, but our previous observation that $[\min L, \max L] = \bigcup_{i=1}^4 S_i$ requires that $3a_{r+1}-1\in [2a_{r+1}+1,3a_{r+1}]$, which is not in $S_1$, $S_2$, or $S_3$, must therefore be in $S_4$, so $4a_{r+1}-1\in L$, and as $2a_{r+1}-1\in [a_{r+1},2a_{r+1}]\subset L$, we find $2a_{r+1}$ is adjacent to $2a_{r+1}-1$, yet $2a_{r+1}$ is already matched to $-a_{r+1}$, which yields a contradiction.

Thus it remains to check $a_{r+1}=n-1$ and $a_{r+1}=1$.
If $a_{r+1}=n-1$, then we've already identified all $2n$ elements of $L$: namely, 
\[ L=[-2a_{r+1},-a_{r+1}]\cup[a_{r+1},2a_{r+1}] = [-2n+2,-n+1]\cup[n-1,2n-2]. \]
As $n \geq 3$, we see that $a_{r+2}=n$ is not adjacent to any vertex, contradicting the assumption that $L$ induces $nK_2$: it cannot be adjacent to any positive label as $n+(n-1) > 2n-2$, and it cannot be adjacent to any negative label as $n+[-2n+2,-n+1]=[-n+2,1]$ is disjoint from $L$.

Finally, we have the case that $a_{r+1}=1$.
We know $[-2,-1]\cup[1,2]\subset L$, which has $2n\geq 6$ elements.
Our previous argument showed $[2a_{r+1}+1,3a_{r+1}]=\{3\}$ is disjoint from $L$, so $3\not\in L$.
But as $a_{r+1} = 1 < n-1$ as $n \geq 3$, our previous argument demonstrates that as 3 is not in $S_1$, $S_2$, or $S_3$, it must be in $S_4$, so $4 \in L$.
But then $-2$ is adjacent to 4, contradicting the fact that $-2$ is already matched with 1.
Hence, we find that assuming $\ispum(nK_2)=4n-4$ yields a contradiction, and thus $\ispum(nK_2)=4n-3$.
\end{proof}

\section{The sum-diameter of a graph}\label{section: sum-diameter}
As integral spum is allowed to use the integers $\Z$ instead of just the positive integers $\Z_+$, we expect $\ispum(G)$ to be lower than $\spum(G)$.
For example, we have
\[\spum(K_3)=6>2=\ispum(K_3)\]
as shown in \cite{singla2021some}.
However, there is no obvious relationship between $\ispum(G)$ and $\spum(G)$, due to $\zeta(G)$ possibly being strictly smaller than $\sigma(G)$, i.e., when the inequality $\zeta(G) \leq \sigma(G)$ is strict.
For example, \cref{theorem: spum C_n} implies $\spum(C_{12}) = 23$, yet $\ispum(C_{12})=25$ as calculated in \cite{singla2021some}, which we independently verified by exhaustive computer search.
In such cases, $\ispum(G)$ is restricted to work with fewer additional isolated vertices, which can make it more difficult to have a lower $\range(L)$.
It would be convenient to remove the restriction that we are simultaneously minimizing the number of additional vertices as well as the range of the labeling, and instead solely minimize the range.
Motivated by this, we introduce the following modification of spum, which we call the sum-diameter.

\begin{definition}\label{definition: uspum}
The \emph{sum-diameter} of a graph $G$, denoted $\uspum(G)$, is the minimum possible value of $\range(L)$ for a set $L$ of positive integer labels, such that the induced sum graph of $L$ consists of the disjoint union of $G$ with any number of isolated vertices.
\end{definition}
The difference between the definition of sum-diameter and spum is that in spum, one is required to use the minimum number of additional isolated vertices possible.

Notice that the sum-diameter is not related to the distance along edges of a graph, though the terminology is consistent with the diameter of a set in a metric space, if one considers the set of labels $L$ to be inside a metric space.

We similarly define the \emph{integral sum-diameter} of a graph $G$, denoted $\iuspum(G)$, by allowing the labels to be arbitrary distinct integers.

Clearly $\uspum(G) \leq \spum(G)$, as any set of labels $L$ from the definition of $\spum(G)$, i.e., on $n+\sigma(G)$ vertices, is a valid labeling for the definition of $\uspum(G)$.
Similarly, $\iuspum(G) \leq \ispum(G)$.
Finally, as desired, any labeling $L$ from the definition of $\uspum(G)$ is a valid labeling for the definition of $\iuspum(G)$, which yields the following proposition.

\begin{proposition}\label{lemma: iuspum <= uspum}
For all graphs $G$, $\iuspum(G) \leq \uspum(G)$.
\end{proposition}

These relations can thus be visually represented by \cref{fig:spum_relations}.
\begin{figure}[htbp]
    \centering
    \begin{tikzpicture}[scale=1,baseline,thick]
        
        \node at (0,0) {isd};
        \node at (2.5,0) {sd};
        \node at (0,2.5) {ispum};
        \node at (2.5,2.5) {spum};
        
        \node at (1.25,0) {$\leq$};
        \node[rotate=90] at (0,1.25) {$\leq$};
        \node[rotate=90] at (2.5,1.25) {$\leq$};
    \end{tikzpicture}
    \caption{The relationships between (integral) spum and (integral) sum-diameter.}
    \label{fig:spum_relations}
\end{figure}
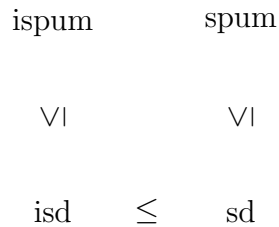

\begin{example}\label{example: iuspum < ispum}
Our previous observation of $\spum(C_{12}) = 23$ while $\ispum(C_{12})=25$ implies that $\iuspum(C_{12}) \leq 23 < \ispum(C_{12}) = 25$, so the leftmost relationship of \cref{fig:spum_relations} can be strict.
This means that adding more isolated vertices than minimally necessary can decrease $\range(L)$ for labelings over the integers $\Z$.
In general, any situation where $\spum(G)<\ispum(G)$ yields $\iuspum(G) < \ispum(G)$.
\end{example}
\begin{example}\label{example: iuspum < uspum}
The middle relationship of \cref{fig:spum_relations} can also be strict.
Take for example $K_2$ and $K_3$; in \cref{lemma: spums equal for K_n 2-3} we prove $\iuspum(K_2)=1<\uspum(K_2)=2$ and $\iuspum(K_3)=2<\uspum(K_3)=6$.
Thus, allowing nonpositive integers can decrease $\range(L)$.
\end{example}
\begin{example}\label{example: uspum < spum}
Lastly, the rightmost relationship of \cref{fig:spum_relations} can also be strict.
For example, consider the path graph with 9 vertices $P_9$.
\cref{tab:spum P_n} shows $\spum(P_9)=19$, yet $\uspum(P_9) \leq 17$ as the labeling $L=[9,17]\cup [25,26]$ induces the sum graph consisting of a path 13, 12, 14, 11, 15, 10, 16, 9, 17, with 25 and 26 as isolated vertices, and $\range(L)=17$.
Hence, adding more isolated vertices than minimally necessary can decrease $\range(L)$ for labelings over the positive integers $\Z_+$.
\end{example}

\begin{remark}\label{remark: a1 = min L sd}
Our observation in \cref{remark: a1 = min L spum} regarding spum also applies to sum-diameter.
If $L$ is an optimal sum graph labeling for a graph $G$ without isolated vertices, i.e., $\range(L)=\sd(G)$, and $S\subset L$ are the labels corresponding to the vertices of $G$, then $\min S = \min L$, as otherwise removing the label $\min L$ does not change the fact that this labeling induces $G$, while strictly decreasing its range, contradicting the minimality of $\range(L)=\sd(G)$.
\end{remark}

The following result was initially stated for $\spum(G)$ rather than $\sd(G)$, but we note that the same proof works to obtain the following stronger result.
\begin{theorem}[\protect{\cite[Theorem 2.1]{singla2021some}}]\label{lemma: uspum lower bound degrees}
For graphs $G$ of order $n$ without any isolated vertices, with maximum and minimum vertex degree $\Delta$ and $\delta$, respectively, we have
\begin{align*}
    \uspum(G) \geq 2n-(\Delta-\delta)-2.
\end{align*}
\end{theorem}

\begin{remark}\label{remark: sd lower bound equality a1-2a1}
Moreover, our refinement in \cref{lemma: spum lower bound equality a1-2a1} of the equality case of this result also generalizes to the equality case $\sd(G)=2n-(\Delta-\delta)-2$, with the same proof.
\end{remark}

Similarly, the proof of \cite[Theorem 2.2]{singla2021some} does not use the assumption that exactly $\zeta(G)$ isolated vertices are present when bounding $\ispum(G)$, and thus the following strengthening holds.
\begin{theorem}[\protect{\cite[Theorem 2.2]{singla2021some}}]\label{lemma: iuspum lower bound degrees}
For graphs $G$ of order $n$ without any isolated vertices, with maximum degree $\Delta$, we have
\begin{align*}
    \iuspum(G) \geq 2n - \Delta - 3.
\end{align*}
\end{theorem}

As noted in \cite{singla2021some}, these two inequalities stated originally in terms of $\spum(G)$ and $\ispum(G)$ are sharp, and as $\sd(G) \leq \spum(G)$ and $\isd(G) \leq \ispum(G)$ for all graphs $G$, we find the same constructions imply that the two bounds of \cref{lemma: uspum lower bound degrees} and \cref{lemma: iuspum lower bound degrees} are sharp.

At the time of writing, no general upper bounds are known for $\spum(G)$ or $\ispum(G)$, i.e., there are no known upper bounds that hold for all graphs.
The primary difficulty in creating such a bound is that the sum number $\sigma(G)$ and integral sum number $\zeta(G)$ are not known for arbitrary graphs.

One way in which the sum-diameter is arguably a better graph property to study is that simple upper bounds on $\sd(G)$ and $\isd(G)$ can be obtained for arbitrary graphs $G$, as \cref{theorem: sd general upper bound} shows.
Before we present the theorem, however, we first define the notion of a Sidon set.

\begin{definition}
A set $S$ of positive integers is a \emph{Sidon set} if whenever $a,b,c,d\in S$ satisfy $a+b=c+d$, we have $\{a,b\}= \{c,d\}$.
\end{definition}

We use Sidon sets to prove the following result.

\begin{theorem}\label{theorem: sd general upper bound}
Let $G$ be a graph with $n$ vertices.
Then $\sd(G) \leq 64n^2 - 64n + 9$.
\end{theorem}
\begin{proof}
For $n=1$, we have $\sd(G)=0$ so the bound holds.

For $n \geq 2$, we provide a general construction for a set of labels $L$ that induces $G=(V,E)$ along with $|E|$ number of isolated vertices, where $|E|$ is the number of edges in $G$.
As $n \geq 2$, we have by the well-known Bertrand's postulate, first proven by Chebyshev, that there exists an odd prime $p$ such that $n \leq p < 2n$ (see, for example, Derbyshire \cite[p. 124]{derbyshire2003prime}).
Erd\H{o}s and Tur\'an \cite{erdos1941problem} showed that there exists a Sidon set with $p$ elements contained in $[1,2p^2]$.
Pick an arbitrary $n$-element subset of such a Sidon set, enumerating them $s_1$ through $s_n$ in increasing order.
Clearly $S=\{4s_i + 1 \mid 1 \leq i \leq n\}$ is also a Sidon set itself.
Enumerating the $n$ vertices 1 through $n$, we label each vertex $i$ with $4s_i + 1$.
For each edge $(i,j) \in E$, add the label $4s_i + 4s_j + 2$ to $L$.
We claim $L$, which has size $|L|=|V|+|E|$, induces $G$ along with $|E|$ isolated vertices.

Our labeling $L$ has ``vertex-type" labels that are all congruent to 1 mod 4, and ``edge-type" labels that are all congruent to 2 mod 4.
Thus, by simply considering their sums modulo 4, the only edges that $L$ can induce are between vertex-type labels.
So it suffices to show that $L$ induces an edge between labels $4s_i+1$ and $4s_j+1$ if and only if $(i,j) \in E$.
By construction, if $(i,j)\in E$, then $4s_i + 4s_j + 2\in L$, so there is an edge between labels $4s_i + 1$ and $4s_j+1$.
For the other direction, suppose $L$ induces an edge between labels $4s_i+1$ and $4s_j+1$ for $i \neq j$, or in other words $4s_i + 4_j + 2 \in L$.
Then as $4s_i + 4s_j + 2 \equiv 2 \pmod{4}$, this is an edge-type label, and was constructed for some edge $(u,v)\in E$ for $u \neq v$.
So
\[ (4s_i + 1) + (4s_j + 1) = (4s_u + 1) + (4s_v + 1),\]
but as $S$ is a Sidon set this implies $\{i,j\} = \{u,v\}$.
Hence $(i,j) \in E$.

Thus $L$ correctly induces $G$ along with $|E|$ isolated vertices.
We observe that
\[ \max L \leq 4s_{n-1} + 4s_n + 2 \leq 8\cdot2p^2 - 2 \leq 16(2n-1)^2 - 2 \]
and $\min L = 4s_1 + 1 \geq 5$, so
\[\sd(G) \leq \range(L) \leq 64n^2 - 64n + 9,\]
as desired.
\end{proof}
\begin{remark}\label{remark: sd general upper bound implies sigma bound}
This result also trivially bounds $\sigma(G) \leq |E|$ by the number of edges of $G$, which redemonstrates a result by \cite{harary1990sum}, though the construction provided in that paper would yield a worse bound for $\sd(G)$, namely a bound exponential in $n$.

Improvements can certainly be made to the coefficient of the quadratic term; for example, much stronger results are known than simply Bertrand's postulate.
Alternatively, more efficient Sidon sets can be used, for example using Singer's construction \cite{singer1938theorem}.
\end{remark}

Using \cref{lemma: iuspum <= uspum}, we have that \cref{theorem: sd general upper bound} implies $\isd(G) \leq 64n^2-64n+9$ as well.

We now demonstrate that this bound is tight up to a constant factor, i.e., an upper bound asymptotically lower than quadratic is impossible.

\begin{theorem}\label{theorem: existence quadratic sd}
Let $f(n)$ be defined as the maximum value of $\sd(G)$ over all graphs $G$ with $n$ vertices.
Then $f(n) \geq \frac{n^2}{4}-O(n\log n) = \Omega(n^2)$.
\end{theorem}
\begin{proof}
Consider the set of unlabeled connected graphs $G$ with $n$ vertices.
The cardinality of this set is easily seen to be at least $(1-o(1))2^{\binom{n}{2}}/n!$.
For each unlabeled connected graph $G$, consider a labeling $L$ that induces $G$ while achieving $\range(L)=\sd(G)$.
The labelings for all of the unlabeled connected graphs must be pairwise distinct.
All of these labelings satisfy $\range(L) \leq f(n)$, and induce a connected graph $G$ with $n$ vertices, along with some additional isolated vertices.
We claim there are at most $2^{2f(n)}$ such labelings.
Consider the set of labels $S\subset L$ of the $n$ vertices of $G$, and enumerate them $a_1,\dots,a_n$.
Notice that $a_1 = \min S = \min L$, as otherwise labels below $a_1$ can be removed while still inducing $G$, and strictly decreasing $\range(L)$, contradicting the minimality of $\sd(G)$.
Then $a_1$ must be adjacent to some label, meaning $a_1+a_i \in L$ for some $i$, and thus
\[ a_1 + \sd(G) = \max L \geq a_1+a_i \geq a_1 + a_2 \geq 2a_1 + 1,\]
so $a_1 \leq \sd(G) - 1$.
Thus
\[ \max L = a_1 + \sd(G) \leq 2\sd(G) - 1 \leq 2f(n).\]
There are only $2^{2f(n)}$ sets of positive integers such that the maximum element is at most $2f(n)$, and as each distinct connected graph with $n$ vertices must have a distinct such labeling, we find $2^{2f(n)} \geq (1-o(1))2^{\binom{n}{2}}/n!$, so
\[ 2f(n) \geq \binom{n}{2}-\log_2(n!) = \frac{n^2}{2} - O(n\log n) \]
using Stirling's approximation, and thus we have $f(n) \geq \frac{n^2}{4}-O(n\log n)$, as desired.
\end{proof}

\begin{remark}\label{remark: most graphs quadratic sd}
In fact, the proof of \cref{theorem: existence quadratic sd} implies that not only does some graph have $\Omega(n^2)$ sum-diameter, but for sufficiently large $n$, nearly all, i.e., $1-o(1)$ proportion, of the graphs on $n$ vertices have $\Omega(n^2)$ diameter.
This follows from the observation that only $o(1)$ proportion of the graphs can have $o(n^2)$ sum-diameter.
This sharply contrasts the current knowledge of bounds on the sum-diameter, upper bounded by the spum, of special families of graphs, e.g. $K_n$, $C_n$, $K_{1,n}$, $K_{n,n}$, and $nK_2$, all of which have been upper bounded linearly.
\end{remark}

\section{Sum-diameter of complete graphs $K_n$}\label{section: uspum K_n}

We find that the (integral) sum-diameters have the same values as the (integral) spum for complete graphs, i.e., $\spum(K_n)=\uspum(K_n)$ for all $n \geq 2$ and $\ispum(K_n)=\iuspum(K_n)$ for all $n \geq 2$.
\begin{proposition}\label{proposition: all spums equal for K_n 5+}
For $n \geq 5$, we have
\[ \iuspum(K_n)=\uspum(K_n)=\ispum(K_n)=\spum(K_n)=4n-6.\]
\end{proposition}
\begin{proof}
From \cite[Theorem 3.2]{singla2021some} we know $\ispum(K_n)=\spum(K_n)=4n-6$.
Hence
\[ \iuspum(K_n) \leq \uspum(K_n) \leq \spum(K_n)=4n-6.\]
It suffices to show that $\iuspum(K_n) \geq 4n-6$.

Our proof is similar to the original proof by Singla, Tiwari, and Tripathi \cite[Theorem 3.2]{singla2021some} that $\ispum(K_n)=4n-6$.
Let $L$ be a labeling of an integral sum graph $G$ consisting of $K_n$ and some isolated vertices.
Let the labels of $K_n$ be $S=\{a_1,\dots,a_n\}$, sorted in increasing order $a_1 < \cdots < a_n$.
As any two $a_i$ and $a_j$ must be adjacent, we have $a_i+a_j \in L$ for all $1 \leq i < j \leq n$.
Furthermore, observe that if $x,x+a_i \in L$ for some $i$ such that $x\neq a_i$, then $x \in S$, because $a_i$ and $x$ are adjacent and thus $x$ must be part of $K_n$.
We will make use of this fact frequently.

Define $A=a_1+\left(S\setminus\{a_1\}\right)$ and $B=a_n+\left(S\setminus\{a_1,a_n\}\right)$.
From Chen \cite[Theorem 1]{chen1996harary} we know that for $n \geq 5$, the sets $A$, $B$, and $S$ are pairwise disjoint.
Hence, $A\cup B \subseteq L\setminus S$.
We now show that $a_i + a_j \in L \setminus S$ for all $1 \leq i < j \leq n$.
Using the prior results about $A$ and $B$, it suffices to show this for the remaining case $1 < i < j < n$.
Suppose otherwise, that $a_i + a_j \in S$ for some $1 < i < j < n$, so that $a_i + a_j = a_k$ for some $1 \leq k \leq n$.
Then $a_n + a_j \in L$ and $(a_n+a_j)+a_i = a_n + a_k \in L$.
As $\zeta(K_n)=2n-3 > 0$, we find $0 \not\in L$; we can assume $a_n > 0$, as otherwise all $a_i < 0$ and we can use $-L$ instead of $L$.
Moreover, $a_n + a_j > 0 + a_i$, so $a_n + a_j \neq a_i$, so by our prior observation this implies $a_n + a_j \in S$, yet $a_n + a_j \in B$ which is disjoint from $S$, a contradiction.

We now consider $B-|a_1|$, and we split into two cases depending on the sign of $a_1$.

If $a_1 > 0$, notice that $B-a_1$ is disjoint from $S$, as $\max S = a_n < a_n + (a_2-a_1) = \min(B-a_1)$.
We now claim $B-a_1$ is disjoint from $L$.
Suppose otherwise: then there exists an $i$ such that $a_n + a_i - a_1 \in L \setminus S$.
Suppose $a_n + a_i - a_1 = x$, for $x \in L$.
Then $x+a_1 = a_n + a_i$, where $x = a_n + a_i - a_1 > a_i > a_1$, so $x$ and $a_1$ are adjacent, but $a_1$ cannot be adjacent to an element in $L\setminus S$.
Note that $B-a_1 \subset [\min L, \max L]$, as
\[ \min L \leq \max S < \min(B-a_1) < \max(B-a_1) < \max B \leq \max L.\]

If $a_1 < 0$, then $B+a_1$ is disjoint from $L$, as otherwise $a_n + a_i + a_1 \in L$ for some $i$.
But then $a_n+a_i \in L\setminus S$ is adjacent to $a_1\in S$, which is forbidden.
Note that $B+a_1 \subset [\min L, \max L]$, as
\begin{align*}
    \min L &\leq a_1 + a_2 < a_1 + a_2 + a_n = \min(B+a_1) < \max(B+a_1) 
    \\ &= a_1 + a_{n-1} + a_n < a_{n-1} + a_n \leq \max L.
\end{align*}

In either case, as $A \cup B \subseteq L \setminus S$, we have
\[ |L| \geq |S|+|A|+|B|=n+(n-1)+(n-2)=3n-3,\]
and as $\left|B-|a_1|\right| = n-2$, we have
\[ \max L - \min L \geq |L| + \left|B-|a_1|\right| - 1 \geq (n+(2n-3)) + (n-2) - 1 = 4n-6.\]
Hence, $\iuspum(K_n) \geq 4n-6$.
\end{proof}

It is known and easy to confirm that $\sigma(K_2)=1$ and $\sigma(K_3)=2$ (see for example \cite{gallian2018dynamic}), whereas $\zeta(K_2)=\zeta(K_3)=0$, which can be seen by the labelings $\{0,1\}$ and $\{0,\pm 1\}$.
From \cite[Lemma 3.1]{singla2021some}, we have $\spum(K_n)=4n-6$ even for $n=2,3$, but $\ispum(K_2)=1$ and $\ispum(K_3)=2$.

\begin{lemma}\label{lemma: spums equal for K_n 2-3}
We have $\uspum(K_n)=\spum(K_n)=4n-6$ for $n=2,3$, as well as $\iuspum(K_2)=\ispum(K_2)=1$ and $\iuspum(K_3)=\ispum(K_3)=2$.
\end{lemma}
\begin{proof}
The spum and ispum values are already known.
To show that $\uspum(K_2)=2$, notice that as $K_2$ has no isolated vertices, at least one additional vertex is needed, yielding at least 3 vertices and thus $\range(L) \geq 2$ for all valid labelings $L$.
A range of 2 can be achieved by $\{1,2,3\}$, and thus $\uspum(K_2)=2$.

To show that $\uspum(K_3)=6$, since we have $\sigma(K_3) = 2$ we need at least two more vertices.
If we use exactly two additional vertices, then the smallest value of $\range(L)$ is simply $\spum(K_3) = 6$.
If we use at least four additional vertices, for a total of at least seven, then $\range(L) \geq 6$.
The only remaining case is using three additional vertices.
With six vertices, we have $\range(L) \geq 5$.
In order to rule out $\uspum(K_3)=5$, we simply need to check the cases when $L=\{a_1,a_1+1,\dots,a_1+5\}$.
If $a_1 \leq 2$, then the degree of $a_1$ exceeds two, while the maximum degree in $K_3$ is two.
If $a_1 \geq 3$, then all edges are incident to $a_1$, as any other edge, say between $a_i$ and $a_j$, would require $a_i + a_j \in L$, but 
\[ a_i + a_j \geq a_2 + a_3 = 2a_1 + 3 \geq a_1 + 6 \geq \max L, \]
so this is impossible.
This cannot yield $K_3$, so $\uspum(K_3) = 6$.

To show that $\iuspum(K_2)=1$, notice that $K_2$ already has two vertices so $\iuspum(K_2) \geq 1$, but $\iuspum(K_2) \leq \ispum(K_2)=1$, so $\iuspum(K_2) = 1$.

Similarly, to show that $\iuspum(K_3)=2$, as $K_3$ itself has three vertices, we have $\iuspum(K_3) \geq 2$, but $\iuspum(K_3) \leq \ispum(K_3) = 2$, so $\iuspum(K_3) = 2$.
\end{proof}

\begin{lemma}\label{lemma: spums equal for K_4}
For $n=4$, we have $\iuspum(K_n)=\uspum(K_n)=\ispum(K_n)=\spum(K_n)=4n-6$.
\end{lemma}
\begin{proof}
From \cite[Theorem 3.1]{singla2021some} we know $\ispum(K_n) \leq \spum(K_n) \leq 4n-6$.
Hence it suffices to show that $\iuspum(K_n) \geq 4n-6$.

The proof is similar to the original proof that $\ispum(K_n)=4n-6$, except $n=4$ requires some additional consideration due to the result from \cite{chen1996harary} assuming $n \geq 5$.
In particular, \cite{chen1996harary} requires that at least three labels in $S$ are of the same sign.
So if at least three of the four labels in $S$ are of the same sign, the same argument as in \cref{proposition: all spums equal for K_n 5+} applies, yielding $\range(L) \geq 4n-6 = 10$.

Otherwise, we have $a_1 < a_2 < 0 < a_3 < a_4$.
Note that 0 cannot be in $L$ as 0 would be adjacent to every other label, yet we must have at least $\zeta(K_4)=\sigma(K_4)=2n-3=5$ isolated vertices.
We need at least 9 total vertices.
If there are exactly 9 vertices in our graph, and the interval $[\min L, \max L]$ must contain $0 \not\in L$, we find $\range(L) \geq 9$ with equality if and only if there are no other gaps, i.e., $L=[x,x+9]\setminus\{0\}$ for some $x<0$.
Replacing $L$ with $-L$ if necessary, we can assume at least 5 of the terms are positive, whereas at least 2 of the terms are negative as $a_1 < a_2 < 0 < a_3 < a_4$.
So $L$ contains $[1,5] \cup [-2,-1]$, meaning 1 is adjacent to 2, 3, 4, and $-2$, yielding degree 4 which is too high.
Thus $\range(L) \geq 10$ in this case.

If we have at least 10 vertices, where $a_1 < a_2 < 0 < a_3 < a_4$, then $[\min L, \max L]$ contains $0 \not\in L$, so $\range(L) \geq 10$.

In all cases, $\range(L) \geq 10$, so $\iuspum(K_4) \geq 10 = 4n-6$.
\end{proof}

Together, these three results imply the following theorem.
\begin{theorem}\label{theorem: unlimited equal for K_n}
We have $\uspum(K_n)=\spum(K_n)=4n-6$ for all $n \geq 2$.
In addition, for all $n \geq 2$ we have $\isd(K_n)=\ispum(K_n) = \begin{cases} n-1 & n \leq 3 \\ 4n-6 & n \geq 4.\end{cases}$
\end{theorem}

\begin{remark}\label{remark: singla error K_n Chen n=4}
In \cite{singla2021some}, Theorem 3.2 states that for $n \geq 4$, we have $\spum(K_n)=\ispum(K_n)=4n-6$.
Their proof cites results from \cite{chen1996harary} claiming that $S$ is sum-free, i.e., $a_i + a_j \in L\setminus S$ for all $1 \leq i < j \leq n$.
However, \cite{chen1996harary} only demonstrates that $A \cup B \subseteq L\setminus S$, i.e., they only address sums $a_i+a_j$ where either $i=1$ or $j=n$.
Thankfully, as shown in the proof of \cref{proposition: all spums equal for K_n 5+}, it is not hard to extend this observation to include all $a_i+a_j$; this has also been addressed in a recent paper by Elizeche and Tripathi \cite{elizeche2020characterization}.

More pressingly, \cite{chen1996harary} assumes $n \geq 5$ in order for at least three labels to be of the same sign, which allows them to conclude that $A \cup B \subseteq L \setminus S$.
Thus, the proof of \cite[Theorem 3.2]{singla2021some} does not hold for $n=4$.
We separately address this case in \cref{lemma: spums equal for K_4}, taking care to only assume $\ispum(K_4)\leq \spum(K_4) \leq 10$ using the explicit construction from \cite[Theorem 3.1]{singla2021some}, rather than using the stronger assertion of \cite[Theorem 3.2]{singla2021some}.

\cref{lemma: spums equal for K_4} resolves the case $n=4$, and thus while the proof of \cite[Theorem 3.2]{singla2021some} may not be entirely correct, the result itself is demonstrably correct.
\end{remark}

\section{Integral sum-diameter of cycles $C_n$ and paths $P_n$}\label{section: isd cycle path}
The current knowledge regarding $\ispum(C_n)$ and $\ispum(P_n)$ is relatively lacking compared to other common families of graphs.
The integral spum of complete graphs $K_n$, star graphs $K_{1,n}$ and complete symmetric bipartite graphs $K_{n,n}$ are all exactly known \cite{singla2021some}.
The spum of cycles was previously bounded within $[2n-2,2n-1]$ by \cite{singla2021some}; however, the current best bounds for the integral spum of cycles are
\begin{align*}
    2n-5 \leq \ispum(C_n) \leq \begin{cases} 8(n-9) & \text{if } n \text{ is odd} \\ \frac{3}{2}(3n-14) & \text{if } n \text{ is even},\end{cases}
\end{align*}
where the upper bound comes from \cref{theorem: ispum cycle upper bound} and holds for $n \geq 12$.
The lower bound is due to \cite{singla2021some} and holds for all $n \geq 4$.

Intuitively, the reason it is harder to close this gap for the integral spum of cycles is that $\zeta(C_n)=0$ for all $n \neq 4$, as shown by \cite{sharary1996integral}, while $\sigma(C_n)=2$ for all $n \neq 4$.
This restriction to not being able to use additional isolated vertices makes constructing an efficient, i.e., low range, labeling $L$ quite difficult.

This is a second example in which the sum-diameter is arguably a better property to study, in addition to there being simple upper bounds on the sum-diameter of arbitrary graphs as provided in \cref{section: sum-diameter}.
First, we bound $\uspum(C_n)$ by the same bounds as for $\spum(C_n)$.
Notice that we assume $n \geq 4$, as when $n=3$ we recover $K_3$, which we have completely determined in \cref{section: uspum K_n}.
\begin{proposition}\label{proposition: trivial C_n uspum bounds}
For all $n \geq 4$, we have
\begin{align*}
    2n-2 \leq \uspum(C_n) \leq 2n-1.
\end{align*}
\end{proposition}
\begin{proof}
The upper bound follows from the fact that $\uspum(C_n) \leq \spum(C_n) = 2n-1$, from \cref{theorem: spum C_n}.
The lower bound follows from \cref{lemma: uspum lower bound degrees}.
\end{proof}

This allows us to bound $\iuspum(C_n)$ far more effectively than the current bounds for $\ispum(C_n)$.
\begin{proposition}\label{proposition: trivial C_n iuspum bounds}
For all $n \geq 4$, we have
\begin{align*}
    2n-5 \leq \iuspum(C_n) \leq 2n-1.
\end{align*}
\end{proposition}
\begin{proof}
The lower bound follows from \cref{lemma: iuspum lower bound degrees}, and the upper bound follows from $\iuspum(C_n) \leq \uspum(C_n) \leq 2n-1$, where the first inequality comes from \cref{lemma: iuspum <= uspum} and the second comes from \cref{proposition: trivial C_n uspum bounds}.
\end{proof}

Similarly, for the path graphs $P_n$, the current best bounds for $\spum(P_n)$ are
\begin{align*}
    2n-2 \leq \spum(P_n) \leq \begin{cases} 2n+1 & \text{if } n \text{ is odd} \\ 2n-1 & \text{if } n \text{ is even}, \end{cases}
\end{align*}
where the lower bound is due to \cref{theorem: spum path lower bound} and holds for $n \geq 7$, and the upper bound is from \cref{theorem: spum path upper bound}.
However, the current best bounds for $\ispum(P_n)$ are
\begin{align}\label{eq: ispum bound P_n}
    2n-5 \leq \ispum(P_n) \leq \begin{cases} \frac{5}{2}(n-3) & \text{if } n \text{ is odd} \\ 2n-3 & \text{if } n \text{ is even}, \end{cases}
\end{align}
again due to \cite{singla2021some}, where the upper bounds were proven for $n \geq 7$.

Despite the integral spum being able to work with a larger set of possible labels, due to $\zeta(P_n)=0$ \cite{harary1994sum} being strictly less than $\sigma(P_n)=1$ \cite{harary1990sum}, it is not necessarily true that $\ispum(P_n) \leq \spum(P_n)$, and without any additional isolated vertices to work with, constructing an efficient labeling $L$ is quite difficult.

However, by using the sum-diameter, we find a much tighter result for integral sum-diameter.
First, we translate \cref{eq: spum bound P_n} to $\uspum(P_n)$.
\begin{proposition}\label{proposition: trivial P_n uspum bounds}
For $n \geq 3$, we have
\begin{align*}
    2n-3 \leq \uspum(P_n) \leq \begin{cases} 2n+1 & \text{if } n \text{ is odd} \\ 2n-1 & \text{if } n \text{ is even}. \end{cases}
\end{align*}
\end{proposition}
\begin{proof}
The upper bound follows from \cref{theorem: spum path upper bound} and the fact that $\uspum(P_n) \leq \spum(P_n)$ for all $n$.
The lower bound follows from \cref{lemma: uspum lower bound degrees}.
\end{proof}
In fact, we can improve the upper bound using a construction with two isolated vertices.
\begin{proposition}\label{proposition: improved P_n sd bounds}
For $n \geq 3$, we have
\begin{align*}
    2n-3 \leq \sd(P_n) \leq 2n-2.
\end{align*}
\end{proposition}
\begin{proof}
The lower bound follows from \cref{proposition: trivial P_n uspum bounds}.
The upper bound follows from the labeling $L=[n-1,2n-2]\cup\{3n-4,3n-3\}$, which induces the path $2n-2, n-1, 2n-3, n, \dots$, alternating between decreasing from $2n-2$ and increasing from $n-1$, where the exact formulation of the other end of the path depends on the parity of $n$.
It is straightforward to check that no other edges exist, as all edges must have their sum of the two labels being either $3n-4$ or $3n-3$.
\end{proof}

An exhaustive computer search yields $\sd(P_n)=2n-3$ for $3 \leq n \leq 6$, but $\sd(P_n)=2n-2$ for $7 \leq n \leq 13$.
This yields the following conjecture.
\begin{conjecture}\label{conjecture: sd P_n 2n-2}
For $n \geq 3$, we have
\begin{align*}
    \sd(P_n) = \begin{cases} 2n-3 & 3 \leq n \leq 6, \\ 2n-2 & n \geq 7. \end{cases}
\end{align*}
\end{conjecture}
In order to address the general case of $n \geq 7$, it suffices to assume $\sd(P_n)=2n-3$, which allows one to use \cref{remark: sd lower bound equality a1-2a1}, and reach a contradiction.

This yields a far more efficient bound for $\iuspum(P_n)$ than the bounds for $\ispum(P_n)$.
\begin{proposition}\label{proposition: trivial P_n iuspum bounds}
For $n \geq 3$, we have
\begin{align*}
    2n-5 \leq \iuspum(P_n) \leq \begin{cases} 2n-2 & \text{if } n \text{ is odd} \\ 2n-3 & \text{if } n \text{ is even}. \end{cases}
\end{align*}
\end{proposition}
\begin{proof}
The lower bound follows from \cref{lemma: iuspum lower bound degrees}.
The upper bound for the odd case follows by applying \cref{lemma: iuspum <= uspum} to \cref{proposition: improved P_n sd bounds}, and the upper bound for the even case when $n \geq 8$ follows from \cref{eq: ispum bound P_n} and the fact that $\iuspum(P_n) \leq \ispum(P_n)$ for all $n$.
For $n\in\{4,6\}$, we have $\isd(P_n) \leq \ispum(P_n) \leq 2n-4 < 2n-3$ by using $L=\{-1,1,2,3\}$ and $L=\{-1,1,3,4,6,7\}$.
\end{proof}

\section{Sum-diameter under various binary graph operations}\label{section: sd binary operations}

In this section we discuss the behavior of the sum-diameter under various binary graph operations.

First, we prove the following lemma, which will be frequently used.
\begin{lemma}\label{lemma: max L bounded 2 sd}
Let $G$ be a graph with no isolated vertices.
Then for a labeling $L$ that induces $G$ along with some additional isolated vertices and achieves $\range(L)=\sd(G)$, we have $\max L \leq 2 \sd(G) - 1$.
\end{lemma}
\begin{proof}
Suppose the labels of $G$ in $L$ are $S$, so that $L\setminus S$ are the labels of the additional isolated vertices.
Let $S=\{a_1,\dots,a_n\}$ where we sort $a_1 < \cdots < a_n$.
As $G$ has no isolated vertices, $a_1$ must be adjacent to some other vertex $a_i$ for some $i \geq 2$, which means $a_1 + a_i \in L$, and thus $a_1 + a_i \leq \max L = a_1 + \sd(G)$, so $a_i \leq \sd(G)$.
As $a_1 \leq a_i - (i-1) \leq \sd(G)-1$, we find $\max L = a_1 + \sd(G) \leq 2\sd(G)-1$.
\end{proof}

\subsection{Disjoint union}

Let $G_1\cup G_2$ denote the disjoint union of graphs $G_1$ and $G_2$.
The following result provides an upper bound on $\sd(G_1 \cup G_2)$ given $\sd(G_1)$ and $\sd(G_2)$.

\begin{theorem}\label{theorem: sd disjoint union upper bound}
Suppose $G_1$ and $G_2$ are two graphs with no isolated vertices.
Then
\begin{align}\label{eq: sd disjoint union upper bound}
    \sd(G_1\cup G_2) \leq 2(2\sd(G_1)-1)(2\sd(G_2)-1) - 1.
\end{align}
\end{theorem}
\begin{proof}
Consider a labeling $L_1$ that induces $G_1$ (along with some isolated vertices) while achieving $\range(L_1)=\sd(G_1)$ and likewise a labeling $L_2$ that induces $G_2$ while achieving $\range(L_2)=\sd(G_2)$.
From \cref{lemma: max L bounded 2 sd}, we know $\max L_1 \leq 2\sd(G_1)-1$ and $\max L_2 \leq 2\sd(G_2)-1$.

Now consider the labeling $L_2' = (4\sd(G_1)-2)L_2$, i.e., every label in $L_2$ is multiplied by $4\sd(G_1) - 2$.
We claim $L_1\cup L_2'$ induces a labeling consisting of $G_1 \cup G_2$ along with some isolated vertices.
Notice that scaling a labeling does not change its induced sum graph, so $L_2'$ induces $G_2$ and $L_1$ induces $G_1$.
Hence it suffices to show that no additional edges exist between $L_1$ and $L_2'$, and no additional edges exist between the vertices in $L_1$ due to the presence of $L_2'$ or vice versa.

As $\sd(G_1) \geq 1$, we have $4\sd(G_1) - 2 \geq 2\sd(G_1) > 2\sd(G_1) - 1$.
So,
\[ \max L_1 \leq 2\sd(G_1)-1 < 4\sd(G_1) - 2 \leq \min L_2'.\]
If an edge existed between $L_1$ and $L_2'$, suppose this edge is between $x_1 \in L_1$ and $y_1 \in L_2'$.
Then $x_1 + y_1 \in L_2'$, so suppose $x_1 + y_1 = y_2$ for $y_2 \in L_2'$.
As $x_1 > 0$ we have $y_2 > y_1$, and in particular $x_1 = y_2 - y_1 \geq 4\sd(G_1)-2 > 2\sd(G_1) - 1$, but $x_1 \leq \max L_1 \leq 2\sd(G_1)-1$, a contradiction.
Thus no edge exists between $L_1$ and $L_2'$.

Clearly no additional edge between two vertices in $L_2'$ exists due to the presence of $L_1$, as $\min L_2' > \max L_1 > 0$.
It remains to show that no additional edge between two vertices in $L_1$ exists due to the presence of $L_2'$.
Suppose otherwise, that there exist distinct $x_1,x_2 \in L_1$ such that $x_1 + x_2 \in L_2'$.
However, we find
\[ x_1 + x_2 \leq 2\max L_1 - 1 \leq 4\sd(G_1) - 3 < 4\sd(G_1) - 2 \leq \min L_2',\]
so $x_1 + x_2\not\in L_2'$, and thus no such edge exists.

We now evaluate $\range(L_1 \cup L_2')$.
As $\max L_1 < \min L_2'$, we have
\begin{align*}
    \range(L_1 \cup L_2') &= (4\sd(G_1) - 2)\max L_2 - \min L_1
    \leq (4\sd(G_1)-2)(2\sd(G_2)-1) - 1
    \\ &= 2(2\sd(G_1)-1)(2\sd(G_2)-1) - 1.
\end{align*}

Hence, $\sd(G_1\cup G_2) \leq 2(2\sd(G_1)-1)(2\sd(G_2)-1) - 1$.
\end{proof}
\begin{remark}\label{remark: sd disjoint union upper bound proof doesn't work for spum}
The proof of \cref{theorem: sd disjoint union upper bound} cannot be modified to work for spum, as while the proof shows $\sigma(G_1 \cup G_2) \leq \sigma(G_1)+\sigma(G_2)$, without knowing the precise value of $\sigma(G_1 \cup G_2)$, this construction may not be a valid labeling for $\spum(G_1 \cup G_2)$ if it does not use the minimum number of additional vertices necessary, i.e., if $\sigma(G_1 \cup G_2) < \sigma(G_1) + \sigma(G_2)$.

If we had $\sigma(G_1 \cup G_2) = \sigma(G_1)+\sigma(G_2)$, then one could replace all of the sum-diameters in \cref{eq: sd disjoint union upper bound} with spums, as the same proof would hold.

This is a third reason, in addition to the ones provided in Section \ref{section: sum-diameter} and \ref{section: isd cycle path}, as to why sum-diameter is arguably a better property to study than spum.
\end{remark}

The previous result relies on the fact that the sum graph structure is invariant under scaling of the labelings.
While the sum graph structure is not invariant under translation, the following lemma demonstrates translation can be a useful construction method to enforce additional structure on a labeling.

\begin{lemma}\label{lemma: translation lemma}
Let $G$ be a graph with no isolated vertices, and let $L$ be a labeling that induces $G$ along with some additional isolated vertices that achieves $\range(L)=\sd(G)$.
Let $S\subset L$ be the set of labels associated with $G$ itself, and let $T$ be the set of labels $c \in L$ such that there exist distinct $a,b \in L$ where $a+b=c$.
Then for any $x \geq \sd(G) - 1 - \min L$, let $L' = (S+x)\cup(T+2x)$.
Then $L'$ induces $G$ along with some additional isolated vertices, so that if for any $c\in L'$ there exist distinct $a,b \in L'$ with $a+b=c$, then $c$ is an isolated vertex.
\end{lemma}
\begin{proof}
If we originally had an edge relation of the form $a+b=c$ in $L$, as $a,b \in S$ and $c \in T$, then we still have this edge relation in $L'$, for $(a+x)+(b+x)=c+2x$.
From \cref{lemma: max L bounded 2 sd} we know $L\subseteq[\min L, \min L + \sd(G)]$, and as $\max L$ must be isolated, we have $\max L \not\in S$, so $S \subseteq [\min L, \min L + \sd(G)-1]$.
Hence $S+x \subseteq [\min L + x, \min L + \sd(G) + x - 1]$.
Similarly, $T \subseteq [2\min L + 1, \min L + \sd(G)]$ yields $T+2x \subseteq [2\min L + 1 + 2x, \min L + \sd(G) + 2x]$.
Notice that $\max (S+x) < \min (T+2x)$ as $x \geq \sd(G)-1-\min L$.
We have
\[ \range(T+2x)\leq \sd(G)-\min L - 1 < \sd(G), \]
so there are no edges between $S+x$ and $T+2x$ in the induced sum graph of $L'$.
There are no edges between two elements of $T+2x$ as the sum of any two such elements strictly exceeds $\max L' = \max T + 2x$, again using the assumption that $x \geq \sd(G)-1-\min L$.
Two elements of $S+x$ are adjacent in the induced sum graph of  $L'$ if and only if their corresponding elements in $S$ are adjacent in the induced sum graph of $L$.
The sum of two adjacent elements in $S+x$ must be in $T+2x$, by construction.
Hence, we have the desired result.
\end{proof}

Notice that this translation construction separates $S$ and $T$, for $\max(S+x)<\min(T+2x)$.
As $S+x$ corresponds to the vertices of $G$ while $T+2x$ consists of isolated vertices that induce the desired edges in $G$, we will often refer to $S+x$ as the set of ``vertex labels" and $T+2x$ as the set of ``edge labels."

The translation method of \cref{lemma: translation lemma} implies the following result.

\begin{lemma}\label{lemma: translation separation sd double}
Let $G$ be a graph with no isolated vertices.
The minimum range of a labeling $L'$ of $G$ that only has edges of the form $(u,v)$ where $u+v$ is an isolated vertex is at most $2\sd(G)-2$.
\end{lemma}
\begin{proof}
Let $L$ be a labeling that induces $G$ along with some additional isolated vertices that achieves $\range(L)=\sd(G)$.
We use the construction of \cref{lemma: translation lemma} with $x = \sd(G) - 1 - \min L$ to yield $L'$.
Notice that $L'$ satisfies the required condition, and we observe
\begin{align*}
    \range(L) &= \max L - \min L = \max T - \min S + x \leq \min L + \sd(G) - \min L + x
    \\ &= \sd(G) + x = 2\sd(G) - 1 - \min L \leq 2\sd(G) - 2,
\end{align*}
as desired.
\end{proof}

The following result now uses translation for a more efficient bound on $\sd(G_1\cup G_2)$ than that provided in \cref{theorem: sd disjoint union upper bound}, though it may use more than $\sigma(G_1)+\sigma(G_2)$ isolated vertices, unlike the previous proof.

\begin{theorem}\label{theorem: sd disjoint union upper bound translation improvement}
Suppose $G_1$ and $G_2$ are two graphs with no isolated vertices.
Then
\begin{align*}
    \sd(G_1 \cup G_2) \leq 10\max\{\sd(G_1),\sd(G_2)\} + \sd(G_1) + \sd(G_2) + 2.
\end{align*}
\end{theorem}
\begin{proof}
As in \cref{theorem: sd disjoint union upper bound}, consider an optimal labeling $L_1$ that induces $G_1$ (along with some isolated vertices) and likewise an optimal labeling $L_2$ that induces $G_2$.
Let the labels associated with $G_1$ and $G_2$ themselves be $S_1\subset L_1$ and $S_2 \subset L_2$.
From \cref{lemma: max L bounded 2 sd}, we know $\max L_1 \leq 2\sd(G_1)-1$ and $\max L_2 \leq 2 \sd(G_2) - 1$.
This is equivalent to $\min L_1 \leq \sd(G_1)-1$ and $\min L_2 \leq \sd(G_2) - 1$.
Without loss of generality, assume $\sd(G_1) \geq \sd(G_2)$.

We apply the translation construction in \cref{lemma: translation lemma} using $x = \sd(G_1)+1-\min L_1$ to yield a labeling $L_1'$ that induces $G_1$.
We now have $\min L_1' = \sd(G_1)+1$.
As the vertex labels were originally in $[\min L_1, \min L_1 + \sd(G_1)-1]$, they are now in $[\sd(G_1)+1,2\sd(G_1)]$.
Similarly, the edge labels were originally in $[2\min L_1 + 1, \min L_1 + \sd(G_1)]$, so they are now in $[2\sd(G_1)+3,3\sd(G_1)+2-\min L_1]$.

We now address $G_2$, which we will incorporate by translating $L_2$ to be above $L_1'$.
We use the \cref{lemma: translation lemma} translation technique with $x=6\sd(G_1)+2-\min L_2 > \sd(G_2)-1-\min L_2$, to yield a labeling $L_2'$ that by itself induces $G_2$.
Notice that as $\max L_1' \leq 3\sd(G_1)+1$, as $\min L_2' = 6\sd(G_1)+2 \geq 2\max L_1'$, no additional edges are created between two vertices in $L_1'$ due to the inclusion of $L_2'$.
Using the same analysis as for $L_1'$, this puts the vertex labels in the interval $[6\sd(G_1)+2,6\sd(G_1)+\sd(G_2)+1]$, and puts the edge labels in the interval $[12\sd(G_1)+5,12\sd(G_1)-\min L_2 + \sd(G_2) + 4]$.

Finally, it remains to show that there are no edges between $L_1'$ and $L_2'$.
The maximum element of $L_1'$ is at most $3\sd(G_1)+1$, and the distance between any vertex label of $L_2'$ and an edge label of $L_2'$ is at least $6\sd(G_1)-\sd(G_2)+4 \geq 5\sd(G_1)+4$, so we never have $a+b=c$ for $a \in L_1'$, $b$ a vertex label of $L_2'$, and $c$ an edge label of $L_2'$.
The minimum element of $L_1'$ is $\sd(G_1)+1\geq \sd(G_2)+1$, which is larger than the range of the interval containing the vertex labels of $L_2'$ as well as the range of the interval containing the edge labels of $L_2'$, so we find that no edge exists between a label in $L_1'$ and a label in $L_2'$.

Thus, $L_1' \cup L_2'$ induces $G_1 \cup G_2$ along with some isolated vertices.
Hence,
\begin{align*}
    \sd(G_1 \cup G_2)
    &\leq \range(L_1'\cup L_2')
    \leq (12\sd(G_1)-\min L_2 + \sd(G_2)+4)-(\sd(G_1)+1) 
    \\ &\leq 11\sd(G_1) + \sd(G_2) + 2,
\end{align*}
and invoking our assumption that $\sd(G_1)\geq \sd(G_2)$ yields the result.
\end{proof}

See \cref{fig: sd disjoint union schematic} for a schematic diagram of the construction used in \cref{theorem: sd disjoint union upper bound translation improvement}.
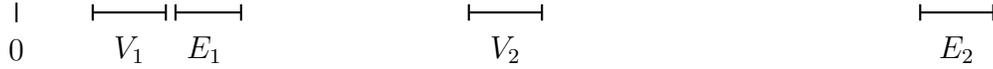
\begin{figure}[htbp!]
    \centering
    \begin{tikzpicture}[scale=1,baseline,thick]
        
        \draw (0,-0.13) -- (0,0.13);
        \node at (0,-0.5) {0};
        
        \draw[|-|] (1,0) -- (2,0);
        \node at (1.5,-0.5) {$V_1$};
        
        \draw[|-|] (2.1,0) -- (3,0);
        \node at (2.5,-0.5) {$E_1$};
        
        \draw[|-|] (6,0) -- (7,0);
        \node at (6.5,-0.5) {$V_2$};
        
        \draw[|-|] (12,0) -- (13,0);
        \node at (12.5,-0.5) {$E_2$};
    \end{tikzpicture}
    \caption{Schematic diagram of the \cref{theorem: sd disjoint union upper bound translation improvement} construction.
    $V_1$ and $E_1$ denote the intervals that contain the vertex and edge labels that induce $G_1$, and similarly $V_2$ and $E_2$ denote the intervals that contain the vertex and edge labels that induces $G_2$.}
    \label{fig: sd disjoint union schematic}
\end{figure}

\begin{remark}\label{remark: sd disjoint union upper bound asymptotic tightness}
The upper bound from \cref{theorem: sd disjoint union upper bound translation improvement} can be tight up to a constant factor.

To see this, consider the infinite family of graphs $C_n \cup C_m$ for $n,m \geq 4$.
From \cref{lemma: uspum lower bound degrees} we know $\sd(C_n \cup C_m) \geq 2(n+m)-2$.
For $n \geq 4$, by \cref{proposition: trivial C_n uspum bounds} we know $\sd(C_n) \leq 2n-1$, so $\sd(C_n \cup C_m) \geq 2(n+m)-2 \geq \sd(C_n)+\sd(C_m)$, and thus for this family of graphs, the upper bound is within a factor of 11 of the true sum-diameter value.
\end{remark}

\subsection{Adding isolated vertices}\label{subsection: sd adding isolated vertices}
\cref{theorem: sd disjoint union upper bound translation improvement} assumes $G_1$ and $G_2$ do not have isolated vertices.
For completeness, if one wishes to add isolated vertices to a graph $G$, the following upper bound holds.

\begin{theorem}\label{theorem: sd add isolated upper bound}
Suppose $G$ is a graph with no isolated vertices, and let $N_k$ be the empty graph on $k$ vertices, so that $G \cup N_k$ is the graph consisting of $G$ with $k$ additional isolated vertices.
Then
\begin{align*}
    \sd(G \cup N_k) \leq \max\{k,4\sd(G)\}+k-5.
\end{align*}
\end{theorem}
\begin{proof}
Consider a labeling $L$ that induces $G$, along with at least one additional isolated vertex, that achieves $\range(L) = \sd(G)$.
From \cref{lemma: max L bounded 2 sd}, we know $\max L \leq 2 \sd(G)-1$, and thus adding additional vertices with labels at least $4\sd(G)-2$ does not add any edges incident to any vertex in $L$.
We now separate into two cases.

Case 1: $k \leq 4\sd(G)$.
If we add isolated vertices with labels in $[4\sd(G)-2,8\sd(G)-4]$, these vertices cannot be adjacent to each other, as their smallest sum is $8\sd(G)-3$, which is larger than any label.
This is a total of $4\sd(G)-1$ possible labels, so if $k-1 \leq 4\sd(G)-1$ we can simply add the $k-1$ labels in $[4\sd(G)-2,4\sd(G)+k-4]$ to $L$, which will induce $G$ along with at least $k$ isolated vertices, namely the $k-1$ isolated vertices we added along with the at least one additional isolated vertex that $L$ must have originally induced, due to $G$ not having any isolated vertices.
As $4\sd(G)+k-4$ is the new largest label and $\min L \geq 1$, the range of our labeling is at most $4\sd(G)+k-5=\max\{k,4\sd(G)\}+k-5$.

Case 2: $k > 4\sd(G)$.
We can add in $k-1$ isolated vertices with labels being $[k-2,2k-4]$, where $k-2 \geq 4\sd(G)-1 > 4\sd(G)-2$, so these vertices are truly isolated and do not add any edges incident to any vertex in $L$.
As $2k-4$ is the new largest label and $\min L \geq 1$, the range of our labeling is at most $2k-5=\max\{k,4\sd(G)\}+k-5$.
\end{proof}
\begin{remark}\label{remark: sd add vertices upper bound asymptotic tightness}
The upper bound from \cref{theorem: sd add isolated upper bound} is tight up to a constant factor.

To see this, notice that $\sd(G\cup N_k) \geq \max\{\sd(G),k+1\}$, so if $\sd(G) \geq k+1$, then 
\[ 5\sd(G\cup N_k) \geq 5\sd(G) \geq \max\{k,4\sd(G)\} + k-5,\]
and if $\sd(G) \leq k$, then
\[ 5\sd(G\cup N_k) \geq 5k+5 \geq \max\{4k,4\sd(G)\}+k+5 \geq \max\{k,4\sd(G)\}+k-5,\]
so the upper bound is within a factor of 5 of the true sum-diameter value.
\end{remark}

\subsection{Graph join}
Let $G_1 + G_2$ denote the join of graphs $G_1$ and $G_2$, which is the disjoint union $G_1 \cup G_2$ together with all the edges joining $V_1$ and $V_2$, i.e., edges where one vertex is from the first vertex set and the other from the second.

Before we address the graph join, however, we first address adding a vertex with arbitrary edges incident to it.
\begin{theorem}\label{theorem: sd add vertex upper bound}
Let $G$ be a graph with no isolated vertices, and let $G'$ be a graph obtained by adding a vertex to $G$ along with any desired edges incident to this new vertex.
Then
\begin{align*}
    \sd(G') \leq 4\sd(G)-1.
\end{align*}
\end{theorem}
\begin{proof}
Let $L$ be an optimal labeling that induces $G$ along with some isolated vertices.
\cref{lemma: max L bounded 2 sd} implies that $\min L \leq \sd(G) - 1$.
We use the translation construction in \cref{lemma: translation lemma} with $x=\sd(G)-\min L$ to yield $L'$, which induces $G$.
The vertex labels of $L'$ are within $[\sd(G),2\sd(G)-1]$ as the vertex labels of $L$ are within $[\min L, \min L + \sd(G)-1]$, and the edge labels of $L$ are within $[2\min L + 1,\min L + \sd(G)]$, so the edge labels of $L'$ are within $[2\sd(G)+1,3\sd(G)-\min L]$.

Now multiply all labels of $L'$ by 2 to yield $L''$, which still induces $G$.
Thus the vertex labels are within $[2\sd(G),4\sd(G)-2]$ and the edge labels are within $[4\sd(G)+2,6\sd(G)-2]$.
Now we add our new vertex with label $b=2\sd(G)+1$.
For every vertex label $a_i \in L''$ that our new vertex is adjacent to in $G'$, we add an edge label $b+a_i$, so that our new vertex is adjacent to the desired vertices in $G$.
As all the labels of $L''$ are even and the new labels are all odd, these new labels do not create new edges between any pair of labels in $L''$.
As all edge labels are distinct from the vertex labels in $L''$, no edge label of $L''$ is adjacent to $b$.
Now it remains to show that all the new edge labels $b+a_i$ are isolated.
Clearly $b$ cannot be adjacent to such an edge label, as $b+b+a_i \geq 6\sd(G)+2$, which is larger than any label.
And any label $\ell \in L''$ cannot be adjacent to $b+a_i$, as this would require $\ell + b + a_i = b + a_j$, which implies $\ell = a_j - a_i \leq 2\sd(G)-2$, but $\ell \geq \min L'' = 2\sd(G)$, so no such $\ell$ exists.
Finally, $b+a_i$ and $b+a_j$ cannot be adjacent, as their sum is larger than any label.

This labeling induces $G'$, so we find $\sd(G') \leq (b+4\sd(G)-2)-2\sd(G) = 4\sd(G)-1$.
\end{proof}

\begin{remark}\label{remark: sd add vertex implies cone addition}
In particular, having our new vertex be adjacent to all original vertices yields $G'=G+K_1$, and thus $\sd(G+K_1)\leq4\sd(G)-1$.
The graph $G+K_1$ is obtained from $G$ by adding a vertex that is connected to all vertices in $G$, which visually adds a ``cone" over $G$.
Repeatedly joining $K_1$ to a graph $n$ times, which is equivalent to $G+K_n$, has been called the $n$th cone over $G$, and its chip-firing dynamics have been studied by Brown, Morrow, and Zureick-Brown \cite{brown2018chip} and then by Goel and Perkinson \cite{goel2019critical}.

In the other extreme case, if we do not wish for this new vertex in $G'$ to have any edges incident to it, we can use \cref{theorem: sd add isolated upper bound} instead for a better bound, namely $\sd(G\cup K_1) \leq 4\sd(G)-4$.
\end{remark}

\cref{theorem: sd disjoint union upper bound translation improvement} had four intervals containing labels, which in increasing order were the vertex labels for $G_1$, the edge labels for $G_1$, the vertex labels for $G_2$, and the edge labels for $G_2$.
Similarly, \cref{theorem: sd add vertex upper bound} also uses four intervals in its translation-based argument, though these intervals were slightly different as they were able to overlap by having different parities: the vertex labels for $G$, the edge labels for $G$, the vertex label for $K_1$, and the edge labels for the join.
We could have instead made the vertex label for the additional vertex sufficiently large, and then the edge labels for the join above that, similar to the construction in \cref{theorem: sd disjoint union upper bound translation improvement}, though this would have yielded a worse bound.
It stands to reason that the join between two arbitrary graphs can be constructed using five intervals containing labels, in the following increasing order: the vertex labels for $G_1$, the edge labels for $G_1$, the vertex labels for $G_2$, the join between $G_1$ and $G_2$, and the edge labels for $G_2$.
See \cref{fig: sd join schematic} for a visual schematic of this construction.
Using this concept, the following result provides an upper bound on $\sd(G_1+G_2)$.

\begin{theorem}\label{theorem: sd join upper bound}
Suppose $G_1$ and $G_2$ are two graphs with no isolated vertices, where without loss of generality $\sd(G_1) \leq \sd(G_2)$.
Then
\begin{align*}
    \sd(G_1 + G_2) \leq 8\sd(G_2) + 11\sd(G_1) - 5.
\end{align*}
\end{theorem}
\begin{proof}
We use the same notation as in \cref{theorem: sd disjoint union upper bound translation improvement}.
We apply the translation construction in \cref{lemma: translation lemma} with $x = \sd(G_1)+\sd(G_2)-\min(L_1)$ to yield a labeling $L_1'$ that induces $G_1$.
Let the set of vertex labels of $L_1'$ be $V_1$, and let the set of edge labels of $L_1'$ be $E_1$.
Similar reasoning as before yields
\[ V_1 \subseteq [\sd(G_1)+\sd(G_2),2\sd(G_1)+\sd(G_2)-1] \]
and
\[ E_1 \subseteq [2\sd(G_1)+2\sd(G_2)+1,3\sd(G_1)+2\sd(G_2)-1]. \]
From our previous argument in \cref{theorem: sd disjoint union upper bound translation improvement}, we needed $\min L_2' \geq 2\max L_1'$ to ensure the addition of $L_2'$ would not induce additional edges between labels in $L_1'$.
Thus we apply the translation construction in \cref{lemma: translation lemma} using $x=6\sd(G_1)+4\sd(G_2)-2-\min L_2$, to get a labeling $L_2'$ that in isolation induces $G_2$.
Similar to $V_1$ and $E_1$, define $V_2$ and $E_2$.
We find
\[ V_2 \subseteq [6\sd(G_1)+4\sd(G_2)-2,6\sd(G_1)+5\sd(G_2)-3] \]
and
\[ E_2 \subseteq [12\sd(G_1)+8\sd(G_2)-3,12\sd(G_1)+9\sd(G_2)-5]. \]
By the same reasoning as in \cref{theorem: sd disjoint union upper bound translation improvement}, we find this labeling $L_1'\cup L_2'$ induces $G_1 \cup G_2$.

We now add in all the edges between a vertex of $G_1$ and a vertex in $G_2$, i.e., all $a_i+b_j$ where $a_i \in V_1$ and $b_j \in V_2$.
Let the set of these labels be denoted $E_{12}$.
It suffices to simply add the entire interval
\[ [\min V_1+\min V_2,\max V_1 + \max V_2] = [7\sd(G_1)+5\sd(G_2)-2,8\sd(G_1)+6\sd(G_2)-4] \supseteq E_{12}.\]
We find that these sets of labels, in increasing order, are $V_1,E_1,V_2,E_{12}$, and $E_2$.

We now confirm this does not add any unwanted edges, i.e., edges not between $V_1$ and $V_2$.
Suppose $E_{12}$ caused an additional edge not incident to a label in $E_{12}$.
As $\min E_{12} \geq 2\max E_1$, the two labels cannot both be from $V_1\cup E_1 = L_1'$.
Thus at least one of the labels must be from $V_2$.
We easily see that the labels cannot both be from $V_2$ as $2\min V_2 + 1 > \max E_{12}$, and any edge between $V_1$ and $V_2$ was already wanted, so it remains to check sums between $E_1$ and $V_2$.
But
\[ \min E_1 + \min V_2 \geq 8\sd(G_1)+6\sd(G_2)-1 > 8\sd(G_1)+6\sd(G_2)-4 \geq \max E_{12},\]
so no such unwanted edges are induced.
Finally, suppose $E_{12}$ caused an additional edge incident to a label in $E_{12}$.
The labels cannot both be from $E_{12}$ as $2\min E_{12}+1 > \max E_2$.
If the sum of the labels of this edge was in $E_{12}$, as the range of the interval containing $E_{12}$ is $\sd(G_1)+\sd(G_2)-2$, the other label incident to this edge must have value at most $\sd(G_1)+\sd(G_2)-2 < \min V_1$, but no such label exists.
Otherwise, the sum of the labels of this edge must be in $E_2$, but the minimum distance between an element of $E_2$ and an element of $E_{12}$ is at least $4\sd(G_1)+2\sd(G_2)+1$, so the label must be at least this value, which excludes $V_1$ and $E_1$ from consideration.
The only remaining possibility is a label in $V_2$ plus a label in $E_{12}$ equals a label in $E_2$.
But
\[ \min V_2 + \min E_{12} \geq 13\sd(G_1)+9\sd(G_2)-4 > 12\sd(G_1)+9\sd(G_2)-5 \geq \max E_2,\]
so this is not possible.

Hence, the addition of $E_{12}$ only induces edges between $V_1$ and $V_2$, and it induces all such edges, as desired.
Our labeling $L$ thus induces $G_1 + G_2$, so
\[ \sd(G_1+G_2) \leq \range(L) \leq 11\sd(G_1)+8\sd(G_2) - 5,\]
proving the result.
\end{proof}

See \cref{fig: sd join schematic} for a schematic diagram of the construction used in \cref{theorem: sd join upper bound}.
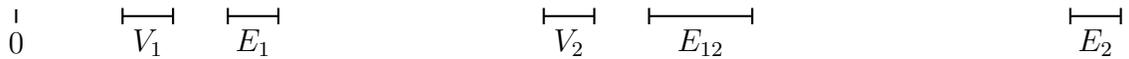
\begin{figure}[htbp!]
    \centering
    \begin{tikzpicture}[scale=0.7,baseline,thick]
        
        \draw (0,-0.13) -- (0,0.13);
        \node at (0,-0.5) {0};
        
        \draw[|-|] (2,0) -- (3,0);
        \node at (2.5,-0.5) {$V_1$};
        
        \draw[|-|] (4,0) -- (5,0);
        \node at (4.5,-0.5) {$E_1$};
        
        \draw[|-|] (10,0) -- (11,0);
        \node at (10.5,-0.5) {$V_2$};
        
        \draw[|-|] (12,0) -- (14,0);
        \node at (13,-0.5) {$E_{12}$};
        
        \draw[|-|] (20,0) -- (21,0);
        \node at (20.5,-0.5) {$E_2$};
    \end{tikzpicture}
    \caption{Schematic diagram of the \cref{theorem: sd join upper bound} construction.}
    \label{fig: sd join schematic}
\end{figure}

\begin{remark}\label{remark: sd join upper bound asymptotic tightness}
The upper bound from \cref{theorem: sd join upper bound} can be tight up to a constant factor.

To see this, consider the infinite family of graphs $K_n + K_m = K_{n+m}$ for $n,m \geq 2$.
From \cref{theorem: unlimited equal for K_n} we know $\sd(K_n + K_m) =4(n+m)-6 = \sd(K_n)+\sd(K_m)+6$.
Thus for this family of graphs, the upper bound is within a factor of $\frac{19}{2}$ of the true sum-diameter value.
\end{remark}

\section{Sum-diameter under vertex and edge operations}\label{section: sd vertex edge addition deletion}
In \cref{section: sd binary operations} we looked at the sum-diameter under various binary graph operations, including the disjoint union and graph join.
In this section we study the sum-diameter under removing or adding a vertex or edge, as well as contracting an edge.

In \cref{theorem: sd add isolated upper bound} and \cref{theorem: sd add vertex upper bound} we already bounded the sum-diameter under adding a vertex with any arbitrary set of incident edges.
We now study the reverse direction, removing a vertex $v$ along with its incident edges from $G$, which we denote $G-v$.

\begin{proposition}\label{proposition: sd delete vertex}
Let $G$ be a graph with no isolated vertices, and let $v$ be a vertex of $G$.
Then
\begin{align*}
    \sd(G-v) \leq 2\sd(G)-2.
\end{align*}
\end{proposition}
\begin{proof}
From \cref{lemma: translation separation sd double}, we find that the optimal range of a labeling of $G$ for which no label is simultaneously both a vertex label and an edge label is at most $2\sd(G)-2$.
For such a labeling, removing the vertex label associated with $v$ thus removes $v$ along with its incident edges from $G$, but as this label is not an edge label, does not change the induced sum graph otherwise.
Hence, a labeling $L$ with $\range(L) \leq 2\sd(G)-2$ exists which induces $G-v$, so $\sd(G-v) \leq 2\sd(G)-2$, as desired.
\end{proof}

In fact, by using the same argument but removing multiple vertices, we have the following result on induced subgraphs of $G$.
Recall that for a subset $U \subseteq V$ of the vertex set $V$ of a graph $G$, the induced subgraph $G[U]$ is the subgraph of $G$ consisting of $U$ and all of the edges connecting pairs of vertices in $U$.
\begin{proposition}\label{proposition: sd induced subgraph}
Let $G$ be a graph with no isolated vertices, and let $U\subseteq V$ be a subset of the vertex set $V$ of $G$.
Then
\begin{align*}
    \sd(G[U]) \leq 2\sd(G)-2.
\end{align*}
\end{proposition}

We present the following stronger statement as an open question.
\begin{question}\label{question: monotonicity wrt induced subgraph}
For any induced subgraph $G[U]$ of $G$, is $\sd(G[U]) \leq \sd(G)$?
\end{question}

Intuitively, the answer to this question seems to be yes: deleting vertices should make the graph simpler, so a labeling with smaller range would be expected.
However, difficulty arises due to the possibility that some of the deleted vertex labels could simultaneously be edge labels, so that removing them would also remove an edge of the induced subgraph.
We observe that \cref{question: monotonicity wrt induced subgraph} holds for complete graphs, and would hold for cycles assuming cycles $\sd(C_n)=2n-1$, i.e., being sharp on the upper bound of \cref{proposition: trivial C_n uspum bounds}, due to the upper bound on $\sd(P_n)$ from \cref{proposition: trivial P_n uspum bounds}.

Now we look at deleting or contracting an edge.
For some edge $e=(u,v)$ in $G$, let $G\setminus e$ denote edge deletion, i.e., the graph $G$ with $e$ removed, and let $G/e$ denote edge contraction, the graph obtained from $G$ by removing $e$ and then identifying $u$ and $v$ together.
\begin{proposition}\label{proposition: sd delete or contract edge}
Suppose $G$ is a graph with no isolated vertices, and let $e=(u,v)$ be an edge of $G$.
Then
\begin{align*}
    \sd(G\setminus e) \leq 4\sd(G)-1
\end{align*}
and
\begin{align*}
    \sd(G/e) \leq 4\sd(G)-1.
\end{align*}
\end{proposition}
\begin{proof}
We first address the $\sd(G\setminus e)$ bound.
Intuitively, our approach will be remove $v$ and then re-insert it with $e$ missing.
Suppose $u_1,\dots,u_k$ are the vertices other than $u$ to which $v$ is adjacent, and let their corresponding vertex labels be $a_1,\dots,a_k$.
We start with the construction as in \cref{theorem: sd add vertex upper bound}, where we add a new vertex $v'$ with label $b=2\sd(G)+1$ that is adjacent to vertices $u_1,\dots,u_k$ by adding edge labels $b+a_1,\dots,b+a_k$.
This gives a new labeling that has range at most $4\sd(G)-1$.
This construction initially translated the labels so that the vertex labels and edge labels are separated, and thus by the same reasoning as \cref{proposition: sd delete vertex}, removing the label associated with vertex $v$ simply removes $v$ from the induced sum graph, as this label is not an edge label so it does not remove any edges not incident to $v$.
Removing this label cannot increase the range, and we find the net effect on the induced sum graph is that we replaced $v$ with $v'$, where $v'$ is adjacent to the same vertices except $u$, so our induced sum graph is $G \setminus e$, and thus $\sd(G\setminus e) \leq 4\sd(G)-1$.

For the $\sd(G/e)$ bound, the proof is essentially identical: we remove both vertices $u$ and $v$, replacing them by a vertex with the desired edges resulting from the edge contraction.
\end{proof}

Similarly, we look at adding an edge.
For two vertices $u,v$ in $G$ that are not adjacent, the graph obtained from $G$ by adding edge $e=(u,v)$ will be denoted $G+e$.
\begin{proposition}\label{proposition: sd add edge}
Suppose $G$ is a graph with no isolated vertices, and let $e=(u,v)$ be an edge between two vertices of $G$ that is not present in $G$.
Then
\begin{align*}
    \sd(G+e) \leq 4\sd(G)-1.
\end{align*}
\end{proposition}
\begin{proof}
The proof is essentially the same as that of \cref{proposition: sd delete or contract edge}, where we instead reinsert $v$ with the extra edge $e$ rather than deleting it.
\end{proof}

\section{The sum-diameter of hypergraphs}\label{section: sd hypergraph}
In this section we generalize the sum-diameter to be defined for $k$-uniform hypergraphs and study some of its basic properties.
Recall that a hypergraph is a generalization of a graph where edges can be incident to an arbitrary number of vertices, and a $k$-uniform hypergraph is a hypergraph where all edges are incident to $k$ (distinct) vertices.
When $k=2$, we recover the definition of a graph, so we will study $k$-uniform hypergraphs for $k > 2$.
As loops, i.e., edges that are incident to only one vertex, cannot occur in such a graph, a simple $k$-uniform hypergraph is simply a $k$-uniform hypergraph where no edges are repeated.

We first generalize the notion of a sum graph to a $k$-sum hypergraph.
\begin{definition}\label{definition: k-sum hypergraph}
A simple $k$-uniform hypergraph $G$ is called a \emph{$k$-sum hypergraph} if there is a bijection $\ell$ from the vertex set $V$ to a set of positive integers $L \subset \Z_+$ such that edge $(v_1,\dots,v_k)\in E$ exists if and only if $\ell(v_1)+\cdots+\ell(v_k)\in L$.
We call $L$ a set of labels for the $k$-sum hypergraph $G$.
We will often not distinguish between the vertices and their respective labels.

Similarly, a simple $k$-uniform hypergraph $G$ is called an \emph{integral $k$-sum hypergraph} if such a bijection exists to a set of integers $L\subset \Z$.
\end{definition}

Conversely, any set of positive integers $L$ induces a $k$-sum hypergraph with vertex set $V=L$, and any set of integers $L$ induces an integral $k$-sum hypergraph.
As before, the vertex with maximum label in a $k$-sum hypergraph must be isolated, so any hypergraph without isolated vertices cannot be a $k$-sum hypergraph.
This leads to the following generalization of the sum number of a graph $G$.

\begin{definition}\label{definition: sum number hypergraph}
The \emph{sum number} of a $k$-uniform hypergraph $G$, denoted $\sigma(G)$, is the minimum number of isolated vertices that must be added to $G$ in order to yield a $k$-sum hypergraph.

Similarly, the \emph{integral sum number} of a $k$-uniform hypergraph $G$, denoted $\zeta(G)$, is the minimum number of isolated vertices that must be added to $G$ in order to yield an integral $k$-sum hypergraph.
\end{definition}

Finally, we generalize the sum-diameter of a graph $G$.
\begin{definition}\label{definition: sum-diameter hypergraph}
The \emph{sum-diameter} of a $k$-uniform hypergraph $G$, denoted $\sd(G)$, is the minimum possible value of $\range(L)$ for a set $L$ of positive integer labels, such that the induced $k$-sum hypergraph of $L$ consists of the disjoint union of $G$ with a nonnegative number of isolated vertices, namely at least $\sigma(G)$ additional isolated vertices.

We similarly define the \emph{integral sum-diameter} of a $k$-uniform hypergraph $G$, denoted $\isd(G)$.
\end{definition}

As before, it is clear that $\zeta(G) \leq \sigma(G)$ for all $k$-uniform hypergraphs $G$, and thus we have the following proposition.
\begin{proposition}\label{proposition: isd <= sd hypergraph}
For all $k$-uniform hypergraphs $G$, we have $\isd(G) \leq \sd(G)$.
\end{proposition}

We now provide a lower bound on $\sd(G)$ for $k$-uniform hypergraphs $G$.
\begin{theorem}\label{theorem: sd hypergraph lower bound}
For $k$-uniform hypergraphs $G$ of order $n$ without any isolated vertices, we have
\begin{align*}
    \sd(G) \geq n + \frac{k(k-1)}{2} - 1.
\end{align*}
\end{theorem}
\begin{proof}
Consider an arbitrary labeling $L$ that induces $G$ along with some isolated vertices and achieves $\range(L) = \sd(G)$.
Then let $S\subset L$ be the set of vertex labels of $L$, i.e., the set of labels that correspond to vertices of $G$ in the induced $k$-sum hypergraph.
Sort $S=\{a_1,\dots,a_n\}$ in increasing order $a_1 < \cdots < a_n$.
Notice that $a_1 = \min L$ as any labels smaller than $a_1$ could be removed while still inducing $G$, and this would decrease $\range(L)$, contradicting the assumption that $\range(L) = \sd(G)$ is minimized.
As no vertices of $G$ are isolated, $a_n$ must have an edge incident to it, and the existence of this edge implies an edge label of value at least
\[ a_n + a_1 + \cdots + a_{k-1} \geq ka_1 + \frac{(k-2)(k-1)}{2} + n-1, \]
where this inequality uses the observation that $a_i \geq a_1 + (i-1)$ for all $i$.
Hence,
\[ \range(L) \geq (k-1)a_1 + \frac{(k-2)(k-1)}{2} + n-1 \geq n + \frac{k(k-1)}{2}-1,\]
as desired.
\end{proof}

Similar to \cref{theorem: sd general upper bound}, we provide an upper bound on $\sd(G)$ for $k$-uniform hypergraphs $G$ using the analogous generalization of Sidon sets to $k$-Sidon sets.
However, first we introduce the necessary definitions for $k$-Sidon sets.
We use the notation of O'Bryant \cite{obryant2004complete}, who provides a comprehensive survey of Sidon sets.
\begin{definition}
A set of integers $\mathcal{A}$ is a $B_k$ set if the coefficients of $\left(\sum_{a \in \mathcal{A}} z^a\right)^k$ are bounded by $k!$. 
\end{definition}
Notice that a $B_2$ set is equivalent to a Sidon set: if $a+b \neq c+d$ when $\{a,b\} \neq \{c,d\}$, then each coefficient is at most 2, as we can have $a+b$ and $b+a$ for a given sum.
So a $B_k$ set is a generalization of a Sidon set, i.e., a $k$-Sidon set as it is otherwise referred to.
\begin{definition}
Let $R_k(k!,n)$ denote the largest cardinality of a $B_k$ set contained in $[1,n]$.
Then $\sigma_k(k!)$ is defined by
\begin{align*}
    \sigma_k(k!) = \lim_{n \to \infty} \frac{R_k(k!,n)}{\sqrt[k]{n}}.
\end{align*}
\end{definition}
The construction of Bose and Chowla \cite{bose1962theorems} shows that $\sigma_k \geq 1$, so conversely there exists a $B_k$ set of cardinality $n$ contained in $[1,f(n)]$ where $f(n)=O(n^k)$.

\begin{theorem}\label{theorem: sd hypergraph upper bound}
For $k$-uniform hypergraphs $G$ of order $n$, we have
\begin{align*}
    \sd(G) = O(k^3 n^k).
\end{align*}
\end{theorem}
\begin{proof}
We provide a general construction for a set of labels $L$ that induces $G=(V,E)$ along with $|E|$ isolated vertices in a similar manner as \cref{theorem: sd general upper bound}.
We use the construction of a $B_k$ set of cardinality $n$ contained in $[1,O(n^k)]$, and enumerate the elements of this set $s_1$ through $s_n$ in increasing order.
Clearly $S=\{k^2s_i + 1 \mid 1 \leq i \leq n\}$ is also a $B_k$ set, so we label the vertices of $G$ using $S$.
For each edge $(v_{i_1},\dots,v_{i_k})\in E$, where each $v_{i_j}$ is labeled with $k^2 s_{i_j} + 1$, add the label $k^2 (s_{i_1} + \cdots + s_{i_k}) + k$ to $L$.
We claim $L$, which has size $|L|=|V|+|E|$, induces $G$ along with $|E|$ isolated vertices.

The reasoning is essentially the same as \cref{theorem: sd general upper bound}.
Our vertex labels are all congruent to 1 mod $k^2$, and the edge labels are all congruent to $k$ mod $k^2$.
By simply considering their sums modulo $k^2$, the only edges that $L$ can induce are between $k$ vertex labels, and thus it suffices to show that $L$ induces an edge between $k$ distinct vertex labels if and only if these $k$ vertices form an edge in $G$.
By construction, if $k$ vertices form an edge in $G$, then these $k$ vertex labels form an edge in the induced $k$-sum hypergraph.
For the other direction, suppose $L$ induces an edge between $k$ distinct vertex labels.
The sum of these $k$ vertex labels is thus an edge label, and was constructed for some edge $e \in E$.
Noting that $e$ is a set of $k$ elements, and as all $k!$ orderings of these $k$ elements yield a sum that equals this edge label, by the definition of a $B_k$ set we know no other sums can equal the value of this edge label, so $e$ must correspond to our $k$ vertex labels.
Essentially, the $B_k$ set condition ensures that no two sums of $k$ distinct vertex labels are the same, so adding the necessary edge label for a given edge does not induce other edges.

Thus $L$ correctly induces $G$ along with $|E|$ isolated vertices.
We observe that $\max L = O(k^3 n^k)$ as each edge label is the sum of $k$ vertex labels, each of which are at most $O(k^2 n^k)$.
Hence, $\sd(G) \leq \range(L) \leq \max L = O(k^3 n^k)$.
\end{proof}

\section{Conclusion and open questions}\label{section: conclusion}
There are many avenues for further research regarding sum-diameter.
One could determine or bound the sum-diameter or integral sum-diameter for other special families of graphs.
In particular, the sum-diameter for path graphs has been bounded to be either $2n-3$ and $2n-2$, and supported by computer data we conjecture it is the latter for all $n \geq 7$, as stated in \cref{conjecture: sd P_n 2n-2}.
The spum of path graphs is similarly bounded within a constant number of options, and a computer search supports our conjecture, \cref{conjecture: spum P_n}, that the current upper bound is tight for $\spum(P_n)$.

In \cref{question: monotonicity wrt induced subgraph}, we also leave as an open question whether the sum-diameter is monotonic with respect to taking induced subgraphs.

While we naturally introduced the integral sum-diameter in addition to the sum-diameter, many of our bounding techniques for sum-diameter cannot immediately generalize to integral sum-diameter.
Thus, the basic properties of the integral sum-diameter are still open to be studied.

Lastly, we generalized and bounded the sum-diameter of $k$-uniform hypergraphs.
The generalization of a Sidon set, or a $B_2$ set, to a $B_k$ set allows us to upper bound the sum-diameter, but requires the $k$-uniform assumption.
It is thus an open question as to how the study of sum-diameter can be extended to hypergraphs that are not $k$-uniform.

\section*{Acknowledgements}
We sincerely thank Amanda Burcroff for her input throughout the research process, especially for her insights regarding Sidon sets.
We would also like to thank Milan Haiman and Maya Sankar for helpful ideas regarding construction techniques for the disjoint union and graph join bounds.
We also thank Ashwin Sah and Michael Ren for helpful input.
We thank Mitchell Lee for his editing feedback.
We would like to thank Prof.\ Joe Gallian for his editing feedback and operation of the Duluth REU program.
This research was conducted at the University of Minnesota Duluth Mathematics REU and was supported, in part, by NSF-DMS Grant 1949884 and NSA Grant H98230-20-1-0009.
Additional support was provided by the CYAN Mathematics Undergraduate Activities Fund.

\bibliographystyle{plain}
\bibliography{ref_spum}

\end{document}